\documentclass{amsart}
\usepackage{amsrefs}
\usepackage{graphicx}

\title{Quantum Collections}
\author{Andre Kornell}
\address{Department of Mathematics, University of California, Berkeley, CA 94720-3840}
\email{kornell@math.berkeley.edu}
\thanks{The research reported here was supported by National Science Foundation grants DMS-0753228 and DMS-1066368.}

\newtheorem{theorem}{Theorem}[section]
\newtheorem{lemma}[theorem]{Lemma}
\newtheorem{proposition}[theorem]{Proposition}
\newtheorem{corollary}[theorem]{Corollary}

\theoremstyle{definition}
\newtheorem{definition}[theorem]{Definition}

\theoremstyle{remark}

\theoremstyle{plain}
\newtheorem*{theorem*}{Theorem}
\newtheorem*{lemma*}{Lemma}
\newtheorem*{proposition*}{Proposition}

\theoremstyle{definition}
\newtheorem*{definition*}{Definition}

\theoremstyle{remark}

\newcommand{\subsetof}{\subseteq}
\newcommand{\To}{\rightarrow}

\newcommand{\Tensor}{\bigotimes}
\newcommand{\tensor}{\otimes}
\newcommand{\iso}{\cong}

\newcommand{\Union}{\bigcup}
\newcommand{\intersect}{\cap}
\newcommand{\suchthat}{\,|\,}

\newcommand{\inv}{^{-1}}
\newcommand{\Hom}{\mathop{Hom}}

\newcommand{\Set}{\mathbf{Set}}
\newcommand{\Wstar}{\mathbf{W^*}}
\newcommand{\pWstar}{\mathbf{pW^*}}
\newcommand{\Cstar}{\mathbf{C^*}}
\newcommand{\cCstar}{\mathbf{cC^*}}

\newcommand{\Dsum}{\bigoplus}

\newcommand{\cWstar}{\mathbf{cW^*}}

\newcommand{\Ast}{\mathop{\scalebox{1.5}{\raisebox{-0.2ex}{$\ast$}}}}

\newcommand{\Ob}{\mathrm{Ob}}
\newcommand{\stensor}{\overline{\otimes}}

\renewcommand{\o}{\vphantom{M}^\circ}

\renewcommand{\H}{\mathcal H}
\newcommand{\R}{\mathcal R}
\newcommand{\B}{\mathcal B}
\newcommand{\M}{\mathcal M}
\newcommand{\N}{\mathcal N}

\newcommand{\K}{\mathcal K}

\renewcommand{\L}{\mathcal L}
\renewcommand{\P}{\mathcal P}
\newcommand{\I}{\mathcal I}

\renewcommand{\S}{\mathcal S}

\newcommand{\CC}{\mathbb C}

\newcommand{\RR}{\mathbb R}
\newcommand{\TT}{\mathbb T}

\newcommand{\NNN}{\mathfrak N}

\newcommand{\Cl}{\mathbf C}
\newcommand{\Dl}{\mathbf D}
\newcommand{\Kl}{\mathbf K}

\begin{document}

\begin{abstract}
We develop the viewpoint that $\o\Wstar$, the opposite of the category of $W^*$-algebras and unital normal $*$-homomorphisms, is analogous to the category of sets and functions. For each pair of $W^*$-algebras $\M$ and $\N$, we construct their free exponential $\M^{\ast \N}$, which in the context of this analogy corresponds to the collection of functions from $\N$ to $\M$. We also show that every unital normal completely positive map $\M \To \N$ arises naturally from a normal state on $\M^{\ast \N}$.
\end{abstract}
\maketitle

\section{Introduction}

We examine the category $\Wstar$ of  $W^*$-algebras and unital normal $*$-homomorphisms. In particular, we develop the viewpoint that its opposite category $\o\Wstar$ is the category of ``quantum collections and functions'', up to a canonical equivalence of categories. We will generally omit the adjective `quantum', and refer to these objects simply as collections. Not even all ``commutative'' collections are collections in the conventional sense.

This viewpoint leads us to a pair of operator-theoretic results. Recall that within $\Set$, the category of sets and functions, the set of all functions from a set $X$ to a set $Y$ is defined up to a canonical bijection by a universal property. We are thusly led to the following:

\begin{theorem}\label{Introduction.A}
Let $\M$ and $\N$ be $W^*$-algebras. There is a $W^*$-algebra $\M^{\ast \N}$, and a unital normal $*$-homomorphism $\varepsilon: \M \To \M^{\ast \N} \stensor \N$, such that for any other $W^*$-algebra $\R$, and unital normal $*$-homomorphism $\pi: \M \To \R \stensor \N$, there is a unique unital normal $*$-homomorphism $\rho: \M^{\ast \N} \To \R$ such that $\pi = (\rho  \stensor 1) \circ \varepsilon$. This \emph{free exponential} $W^*$-algebra $\M^{\ast \N}$ is unique up to a canonical isomorphism.
\end{theorem}

Let's write $\o\M$ to denote the $W^*$-algebra $\M$ considered as an object of $\o\Wstar$. Thus, we imagine that up to a canonical equivalence of categories, $\o \M$ is a collection. In this sense, $\o \M^{\ast \N}$ is the collection of functions from $\o \N$ to $\o\M$, and $\o \M^{\ast \N} \stensor \N$ is the Cartesian product of $\o \M^{\ast \N}$ and $\o \N$. Applying the same notation to the morphisms of $\o \Wstar$,  we may call $\o \varepsilon:  \o \M^{\ast \N} \stensor \N \To \o \M$ the evaluation function. 
The notation $\M^{\ast \N}$ is justified by the fact that $\M^{\ast \CC^n}$ is the $n$-fold free power of $\M$.

We will sometimes think of each $W^*$-algebra as being the algebra of observables of some quantum system. Following Kraus \cite{Kraus83}, we will then think of each unital normal completely positive map $\M \To \N$ as corresponding to a quantum operation in the opposite direction. We'll thus be led to view $\o \pWstar$, the opposite of the category of $W^*$-algebras and unital normal completely positive maps, as the category of quantum systems and quantum operations. 

We can interpret this model from our central viewpoint by saying that if $\M$ is the $W^*$-algebra of observables of some quantum system, then $\o \M$ is the collection of all possible configurations of that quantum system. Every $*$-homomorphism is completely positive, so each morphism of $\o \Wstar$ corresponds to a quantum operation; as a function, it assigns a configuration of the codomain system to each configuration of the domain system. Thus, the morphisms of $\o \Wstar$ correspond to \emph{deterministic} quantum operations.

We give two arguments for the position that, in general, a quantum operation is a \emph{probabilistic} assignment of a final configuration to each initial configuration. The first of these arguments begins with the observation that each set is a quantum collection via the correspondence $X \mapsto \o \ell^\infty(X)$. Given any pair of sets $X$ and $Y$, a morphism $X \To Y$ in $\o \Wstar$ is just a function in the ordinary sense, whereas a morphism $X \To Y$ in $\o \pWstar$ assigns a probability distribution on $Y$ to each element of $X$.
The second argument for this position relies on the following operator-theoretic result:

\begin{theorem}\label{Introduction.B}
Let $\M$ and $\N$ be $W^*$-algebras. Let $\psi: \M \To \N$ be a unital normal completely positive map. There exists a normal state $\mu: \M^{\ast \N} \To \CC$ such that $\psi = (\mu \stensor 1) \circ \varepsilon$.
\end{theorem}

A normal state $\mu: \M^{\ast \N} \To \CC$ is the same thing as a probability distribution on $\o\M^{\ast \N}$, the collection of all functions from $\o \N$ to $\o \M$. Formally, this begs the question, but the analogy between states and measures that we've adapted here is basic to noncommutative mathematics. The collection $\o \CC$ is a one-point set, so the quantum operation $\o \mu: \o \CC \To \o \M^{\ast \N}$ randomly selects a configuration of $\o \M^{\ast \N}$. Thus, any quantum operation $\o \psi$ in $\o \pWstar$ can be implemented by randomly selecting a deterministic quantum operation, and then applying it. This statement is false in the ordinary sense, as the \emph{set} of deterministic quantum operations may be empty.

If $\M$ is the direct sum of $\M_0$ and $\M_1$, we say that $\o \M_0$ and $\o \M_1$ are subcollections of $\o \M$. Every collection has a unique maximum subset, and the unique maximum subset of $\o \M^{\ast \N}$ is indeed the set of all functions from $\o\N$ to $\o \M$. Our position is that the collection $\o \M^{\ast \N}$ contains other functions from $\o \N$ to $\o \M$, which cannot be considered individually. Thus, Theorem \ref{Introduction.B} serves to relate the paper's central viewpoint to the idea that unital normal completely positive maps correspond to probabilistic functions, i.e., to probabilistic quantum operations.

It's my pleasure to record my gratitude to Sridhar Ramesh and Dmitri Pavlov, who have patiently answered my questions about category theory and operator theory, respectively. I have also been a beneficiary of their activity in the mathematical community. This paper would not have been possible without the guidance and support of my advisor, Marc Rieffel, whose unwavering encouragement has become an invaluable lesson for this young researcher.

Our reasoning will not require any results beyond the elementary theories of categories and operator algebras, because we will prove some results anew. In particular, the papers of Guichardet \cite{Guichardet66} and Dauns \cite{Dauns72} examine the category $\Wstar$, but we will reestablish its basic properties without referring to these papers. The papers of Tambara \cite{Tambara90} and So\l tan \cite{Soltan06} are other notable predecessors; they construct analogous exponential spaces, but both assume finite-dimensionality. I thank Alexandru Chirvasitu for pointing these papers out to me, and apologize to those authors whose relevant research remains unknown to me. I also thank Alexandru Chirvasitu and Sander Uijlen for pointing out a typo in the proof of Theorem \ref{Free.A}.

I have cited other papers, whose focus on the intersection of logic and quantum theory make them more kin than predecessors to the present paper. Specifically, the papers of Weaver \cite{Weaver10}, and of Heunen, Landsman and Spitters \cite{HeunenLandsmanSpitters07} \cite{HeunenLandsmanSpitters09}, have guided my thinking on what might be termed quantum first-order logic.

Sections 2 and 3 provide the reader with minimal background on category theory and noncommutative mathematics, respectively. Section 4 presents the desired intuition for the opposite category $\o \Wstar$. Section 5 develops the elementary properties of the category $\Wstar$, and concludes by showing that $\Wstar$ does not have all coexponentials in the usual sense. Section 6 defines the categorical tensor product, and concludes similarly, by showing that $\Wstar$ fails to have all coexponentials relative to this tensor product. Section 7 reviews the usual ``spatial'' tensor product. Section 8 contains a number of lemmas, which essentially provide bounds on the size of $\M^{\ast \N}$. Section 9 defines the free exponential $\M^{\ast \N}$, and discusses its basic properties. Section 10 provides the reader with minimal background on completely positive maps. Section 11 contains the proof of Theorem \ref{Introduction.B}. Section 12 discusses measurement from the viewpoint of collections and functions. Section 13 is an appendix, which contains a proof of Stinespring's Theorem \cite{Stinespring55} for normal completely positive maps.

This paper contains far more material than necessary for the proof of Theorem \ref{Introduction.B}. The busy, knowledgeable reader may prefer to read just Lemma 9.2, Theorem 10.1, and Lemma 12.3. Much of the material is intended as an extensive introduction, not only to this paper, but also to two others that the author plans to write, on the tensor exponential $\M^{\stensor \N}$, and on the internal logic of $\o \Wstar$.

\section{Background: Category Theory}
The results that follow are stated in the language of category theory, and are motivated by it. This section reviews the most essential definitions; it is assumed that the reader has already seen limits, functors and natural transformations. The first chapter of MacLane and Moerdijk's \emph{Sheaves in Geometry and Logic} \cite{MacLaneMoerdijk92} is recommended as a compact, but thorough summary. The category $\Set$ of sets and functions showcases the desired intuition for the definitions that follow.

The product of an indexed family of objects $\{X_\alpha\}$  is an object $X$ together with an indexed family  of morphisms $\{p_\alpha: X \To X_\alpha\}$ that is universal among such cones, in the sense that if $Y$ is another object together with an indexed family of morphism $\{f_\alpha: Y \To X_\alpha\}$, then there is a unique morphism $f: Y \To X$ such that $p_\alpha \circ f = f_\alpha$ for all $\alpha$. Like other objects defined by a universal property, the product of $\{X_\alpha\}$ is unique up to a canonical isomorphism. In $\Set$, the product is the Cartesian product together with its projection functions.

In a category with finite products, the exponential of two objects $Y$ and $Z$ is an object $Z^Y$ together with a morphism $e: Z^Y \times Y \To Z$ that is universal among such pairs, in the sense that if $X$ is another object  and $f: X \times Y \To Z$ is another morphism, then there exists a unique map $\lambda f: X \To Z^Y$ such that $e \circ (\lambda f \times 1_Y) = f$. Here, $\lambda f \times 1_Y: X \times Y \To Z^Y \times Y$ is the morphism obtained by apply the defining universal property of products to the maps $f \circ p_X$ and $p_Y$, where $X\times Y$ together with the maps $p_X$ and $p_Y$ is the product of $X$ and $Y$. In $\Set$, $Z^Y$ is the set of all functions from $Y$ to $Z$, $e$ is the evaluation function, and $\lambda f$ is the function $f$ that has been curried on its first argument, i.e., $\lambda f: x \mapsto f(x, \cdot) $.

A category with all finite products and all exponentials, is said to be Cartesian closed. In a Cartesian closed category, the functor $ - \times Y$ has right adjoint $(-)^Y$, which means that there is a natural isomorphism between $\Hom(X \times Y, Z)$ and $\Hom (X, Z^Y)$ as functors in $X$ and $Z$, and in fact, also in $Y$. This isomorphism is given by $f \mapsto \lambda f$.

Given a category $\Cl$, its opposite $\o\Cl$ is the category whose objects and morphisms are those of $\Cl$, but whose morphisms are formally reversed in the sense that $\Hom_{\o\Cl}(X, Y) = \Hom_{\Cl}(Y,X)$. The dual of any universal construction is obtained by performing it in the opposite category; for example, the coproduct of objects $X$ and $Y$ in $\Cl$ is the product of $X$ and $Y$ in $\o\Cl$. In $\Set$, the coproduct of $X$ and $Y$ is their disjoint union.

Two categories $\Cl$ and $\Dl$ are equivalent in case there are functors $F: \Cl \To \Dl$ and $G: \Dl \To \Cl$ such that $G \circ F$ is naturally isomorphic to the identity on $\Cl$ and $F \circ G$ is naturally isomorphic to the identity on $\Dl$. The category $\Cl$ is said to be dual to $\Dl$, if it's equivalent to $\o \Dl$.

\section{Background: Noncommutative Mathematics}

The category $\cCstar$ of commutative unital $C^*$-algebras and unital $*$-homomorphisms is dual to the category $\Kl$ of compact Hausdorff spaces and continuous functions. In one direction, we have the contravariant functor $C: \Kl \To \cCstar$ such that if $X$ is a compact Hausdorff space then $C(X)$ is the $C^*$-algebra of continuous complex-valued functions on $X$, and if $f: X \To Y$ is a continuous function then $C(f): C(Y) \To C(X)$ is precomposition with $f$. In the other direction, we have the contravariant functor $\sigma: \cCstar \To \Kl$ such that if $A$ is a commutative unital $C^*$-algebra then $\sigma(A)$ is the Gelfand spectrum of $A$, and if $\pi: A \To B$ is a unital $*$-homomorphism then $\sigma(\pi): \sigma(B) \To \sigma (A)$ is precomposition with $\pi$.

If $\Cstar$ is the category of all unital $C^*$-algebras, then the above shows that $\Kl$ is equivalent to a full subcategory of $\o\Cstar$. Noncommutative mathematics views $\o\Cstar$ as a category of generalized compact Hausdorff spaces. This generalization is justified in part by quantum theory, in which observables are identified with self-adjoint operators in a characteristically noncommutative operator algebra. Indeed, the self-adjoint elements of a $C^*$-algebra $A$ of norm at most $1$ are in bijective correspondence with morphisms in $\o \Cstar$ from $A$ to $C[-1,1]$, which we identify with the compact Hausdorff space $[-1,1]$. Noncommutative mathematics studies analogous generalizations of familiar classes of objects.

The present paper focuses on the category $\Wstar$ of $W^*$-algebras, and unital normal $*$-homomorphisms. A $W^*$-algebra $\M$ is a unital $C^*$-algebra which, as a Banach space, is the dual of a Banach space $\M_*$. The Banach space $\M_*$ is unique, and is called the predual of $\M$. A unital $*$-homomorphism $\pi: \M \To \N$ between $W^*$-algebras is normal in case it is continuous with respect to the $w^*$-topologies on $\M$ and $\N$.

The category $\cWstar$ of commutative $W^*$-algebras and unital normal $*$-homomorphisms is dual the category $\mathbf{SLM}$, which we presently describe. An object of $\mathbf{SLM}$ is a measure space that, up to a set of measure zero, is the disjoint union of copies of $\{0\}$ with any multiple of the counting measure and copies of $[0,1]$ with any multiple of Lebesgue measure. A morphism of $\mathbf{SLM}$ is just a function, defined up to a set of measure zero, for which the inverse image of a measure zero set is measure zero. The proof of this duality is straightforward from the basic theory; Theorem 16.7 of Lurie's notes \cite{Lurie11} gets us most of the way there. Note that the construction of a measure space for each commutative $W^*$-algebra does involve a choice of measures, so strictly speaking we only obtain a ``weak'' duality.

The duality between $\cWstar$ and $\mathbf{SLM}$ is revealing in one way, but misleading in another. For example, $\RR$ with Lebesgue measure is isomorphic to $\RR$ with Gaussian measure. Similarly, $\RR$ with the Dirac measure at $0$ is isomorphic to $\{0\}$ with counting measure.  In this sense, the objects of $\mathbf{SLM}$ are neither measure spaces, nor Borel spaces.

The terms `von Neumann algebra' and `$W^*$-algebra' are often used interchangeably. When a distinction is made, as it is in the present paper, a von Neumann algebra is a $W^*$-algebra of bounded operators on a Hilbert space. Every $W^*$-algebra has a canonical faithful unital normal $*$-representation, namely the universal normal representation, and the morphisms of von Neumann algebras are typically taken to be the same as those of $W^*$-algebras, so the two categories are equivalent. 

The universal normal representation of a $W^*$-algebra $\M$  is the direct sum of all the GNS representations of $\M$ for its normal states. A nicer choice of canonical representation would the standard form \cite{Haagerup75}, but we choose to avoid this more sophisticated notion. All unital normal $*$-representations of $\M$ are essentially equivalent (Lemma \ref{Completely.B}), so our choice of faithful representations is doubly immaterial.

The present paper uses the term `$W^*$-algebra' to emphasize that it is insensitive to the additional structure of von Neumann algebras. The unqualified term `morphism' always means a morphism of $\Wstar$, i.e., a unital normal $*$-homomorphism. Similarly, all states are understood to be normal, and all representations normal and non-degenerate.

\section{$W^*$-Algebras as Collections}

We presuppose the viewpoint that $\o \Wstar$ is equivalent to a category of set-like objects, the ``quantum collections'' of the title. For conciseness, we will refer to the objects and morphisms of $\o \Wstar$ as \emph{collections} and \emph{functions} respectively.

Formally, the objects and morphisms of $\o \Wstar$ are the same as those of $\Wstar$, so the subtle distinction between these categories can be a source of mental drag. We adopt two strategies to mitigate this effect. First, all formal mathematics will be done in terms of the category $\Wstar$. Second, we will use the notation $\o \M$ to refer to a $W^*$-algebra $\M$ as an object of $\o \Wstar$, and similarly for the morphisms of $\o \Wstar$. Thus, $M_2(\CC)$ is a $W^*$-algebra, whereas $\o M_2(\CC)$ is a collection. We define the disjoint union of a family of collections to be the direct sum of the corresponding $W^*$-algebras.

Every set is a collection via the functorial identification $X \mapsto \o \ell^\infty(X)$. However, not even every commutative collection is a set; for example, if $\lambda$ denotes Lebesgue measure, then $\o L^\infty(\RR, \lambda)$ is commutative, but is not a set. Each commutative collection $\o \M$ is the disjoint union of copies of $\o \CC$, and of $\o L^\infty(\RR, \lambda)$. Note that $\o \CC$ is just a one element set, and that there are no functions from $\o \CC$ to $\o L^\infty(\RR, \lambda)$; the collection $\o L^\infty(\RR, \lambda)$ is the native continuum. The complete Boolean algebra $\P(L^\infty(\RR, \lambda))$ is used in set theory to force the existence of a ``random'' real, so $\o L^\infty(\RR, \lambda)$ may be thought of as a collection of random real numbers. 

If $A$ is a unital $C^*$-algebra, then its second dual $A^{**} $ is canonically a $W^*$-algebra, called the enveloping $W^*$-algebra of $A$. We think of $\o A^{**}$ as the collection of points of the noncommutative compact Hausdorff space $\o A$. Every unital $*$-homomorphism of unital $C^*$-algebras $\pi: A \To B$ extends uniquely to a unital normal $*$-homomorphism $\pi^{**}: A^{**} \To B^{**}$; thus, every continuous function between noncommutative compact Hausdorff spaces is given by a function between their collections of points. 

No analogous result is available for noncommutative sets, which we may define as disjoint unions of collections of the form $\o \B(\H)$. The atomic representation yields a noncommutative point set for every noncommutative compact Hausdorff space, but not every continuous function between noncommutative compact Hausdorff spaces defines a function between their noncommutative point sets. This phenomenon is related to the existence of non-diagonalizable operators.

If $X$ is a compact Hausdorff space in the usual sense, we define $L^\infty(X) = C(X)^{**}$. The collection $\o L^\infty(X)$ is the colimit of $\o L(X,\mu)$ for $\mu$ a finite regular measure on $X$. In particular, the collection $\o L^\infty[-1,1]$ is the disjoint union of the set $[-1,1]$ with uncountably many copies of the continuum $\o L^\infty(\RR, \lambda)$. Thus, from the point of view of the category $\o \Wstar$, the usual compact Hausdorff space $[-1,1]$ is missing some fuzzy points. If the $\o L^\infty[-1,1]$ seems unreasonably large, the reader may be reassured by the fact that the self-adjoint operators of a $W^*$-algebra $\M$ of norm at most $1$ are in canonical bijective correspondence with functions from  $\o\M$ to $\o L^\infty[-1,1]$. 

Every von Neumann algebra is the commutant of a group of unitary operators. In this sense, each object $\o \M$ of $\o \Wstar$ can be thought of as a Hilbert space of quantum states modulo a group of symmetries. To indicate this point of view, we will refer to $\o \M$ as a \emph{quantum system}. The $W^*$-algebra $\M$ is then the algebra of observables of this quantum system, and the affine space $S_\M$ of the normalized positive elements in $\M_*$ is its space of states. As a collection, $\o \M$ is the collection of all possible configurations. We use the word `configuration' in an informal, intuitive sense. There is no class of objects each of which may be called a configuration; there is only a collection of configurations.

If $\o \M$ and $\o \N$ are quantum systems, then a morphism $\o \pi: \o \N \To \o \M$ is just a function that takes each configuration of $\o \N$ to a configuration of $\o \M$. Thus, we say that $\o \pi$ is a \emph{deterministic quantum operation}. We imagine that each deterministic quantum operation may be implemented by a physical process, which beginning with an instance of the system $\o \N$ produces and instance of the system $\o \M$. The class of all quantum operations will be defined in Section \ref{Quantum Operations}.

\section{The Category of $W^*$-Algebras}

This section reviews the construction of limits and colimits in $\Wstar$. 
\begin{proposition}
The category $\Wstar$ has all small products. 
\end{proposition}

\begin{proof}
Let $\{\M_\alpha\}_{\alpha \in I}$ be an indexed family of $W^*$-algebras.  Define 
$$
\Dsum_\alpha \M_\alpha\  =\left\{ m: I \To \Union_\alpha \M_\alpha \,\middle|\, m(\alpha)\in \M_\alpha, \, \sup_\alpha \|m(\alpha)\| < \infty\right\},$$ and
$$
\left(\Dsum_\alpha \M_\alpha\right)_*  = \left\{ \mu: I \To \Union_\alpha \M_{\alpha*} \,\middle|\, \mu(\alpha)\in \M_{\alpha*}, \, \sum_\alpha \|m(\alpha)\| < \infty\right\}.
$$

It is routine to verify that $\Dsum_\alpha \M_\alpha$ is a $C^*$-algebra with coordinatewise operations and $\|m\| = \sup_\alpha \|m(\alpha)\|$ for all $m \in \Dsum_\alpha \M_\alpha$, that $(\Dsum_\alpha \M_\alpha)_*$ is its predual, and that the projections $\pi_\alpha: m \mapsto m(\alpha)$ are $w^*$-continuous.

If $\rho_\alpha: \N \To  \M_\alpha $ is a family of morphisms then $\rho(n)(\alpha) = \rho_\alpha(n)$ defines a morphism $\pi: \N \To \Dsum \M_\alpha$ such that $\pi_\alpha \circ \rho = \rho_\alpha$. Such a morphism is necessarily unique because the operators $m(\alpha)$ uniquely determine any given $m \in \Dsum \M_\alpha$.
\end{proof}

If $\{\M_\alpha\subsetof \B(\H_\alpha)\}$ is a family of von Neumann algebras, then $\Dsum_\alpha \M_\alpha$ can be equivalently defined as the von Neumann algebra generated by the operators in $\Union_\alpha \M_\alpha$ acting on the $\ell^2$-direct sum $\Dsum_\alpha \H_\alpha$.

\begin{definition}
The product of a family $\{\M_\alpha\}$ of $W^*$-algebras is their \emph{direct sum}, and we will denote it by $\bigoplus_\alpha \M_\alpha$.
\end{definition}

\begin{proposition}
The category $\Wstar$ has all small limits.
\end{proposition}

\begin{proof}
Since $\Wstar$ has all small products, it's sufficient to show that $\Wstar$ has equalizers \cite{MacLane71}. To that end, let $\pi_0, \pi_1 : \M \To \N $ be two morphisms, and define $\M_= = \{m \in \M \suchthat \pi_0(m) = \pi_1(m)\}$, a $W^*$-subalgebra of $\M$, with inclusion $\iota: \M_= \To \M$. By construction, $\pi_0 \circ \iota = \pi_1 \circ \iota$. If $\rho : \R \To \M$ also satisfies this property, then $\rho(\R) \subsetof \M_=$, so $\rho$  factors uniquely through the inclusion $\iota$.  
\end{proof}

A number of the constructions that follow imitate the universal representation construction. The next lemma isolates a key component of these proofs.

\begin{lemma}\label{TheCategory.C}
Let $\kappa$ be a cardinal number. Each $W^*$-algebra generated by $\kappa$ many elements has a faithful representation on a Hilbert space of dimension at most $2^{\aleph_0 \cdot \kappa}$.
\end{lemma}

\begin{proof}
Let $\M$ be a $W^*$-algebra generated by $\kappa$ elements. If $\mu$ is a normal state on $\M$, then the GNS Hilbert space $\H_\mu$ is the closed span of $\aleph_0 \cdot \kappa$ vectors, and so has dimension at most $\aleph_0 \cdot \kappa$. Furthermore, each normal state is a function $\M \To \CC$, and so is uniquely determined by its values on all words in the chosen generators and their adjoints; therefore, $\M$ has at most $(2^{\aleph_0})^{ \aleph_0 \cdot \kappa} = 2^{\aleph_0 \cdot \kappa}$ normal states. Taking the direct sum of normal GNS representations, we find that $\M$ has a faithful unital normal $*$-representation on a Hilbert space of dimension at most $2^{\aleph_0 \cdot \kappa}$.
\end{proof}

\begin{proposition}
The category $\Wstar$ has all small coproducts. 
\end{proposition}

\begin{proof}
Let $\{\M_\alpha\}$ be a family of $W^*$-algebras. Let's define a generating cocone from $\{\M_\alpha\}$ to be a family of morphisms $s= \{\iota_\alpha^s: \M_\alpha \To \N_s\}$ into some $W^*$-algebra $\N_s$ that's generated by $\Union_\alpha \iota^s_\alpha(\M_\alpha)$. By Lemma \ref{TheCategory.C} above, we can form the set $S$ of all generating cocones from $\{\M_\alpha\}$, up to the obvious notion of isomorphism. For each $\alpha$, define $\iota_\alpha: \M_\alpha \To \Dsum_{s \in S } \N_s$ by $m_\alpha \mapsto \Dsum_s \iota^s_\alpha(m_\alpha)$, and define $\M$ to be the subalgebra of $ \Dsum_s \N_s$ generated by $\Union_\alpha \iota_\alpha(\M_\alpha)$.

Let $\ \R$ be a $W^*$-algebra, and $\{\rho_\alpha: \M_\alpha \To   \R\}$ a family of morphisms. Let $\N$ be the $W^*$-subalgebra of $\R$ generated by $\Union_\alpha \rho_\alpha(\M_\alpha)$, so that generating cocone $t= \{\rho_\alpha: \M_\alpha \To \N\} $ is an element of $S$. If $\pi_t: \bigoplus \N_s \To \N_t = \N$ is projection onto the $t$ summand, then $\pi_t \circ \iota_\alpha = \iota_\alpha^t = \rho_\alpha$ for all $\alpha$ as desired. Furthermore if $\pi: \M \To \R$ is another morphism such that  $\pi \circ \iota_\alpha  = \rho_\alpha$ for all $\alpha$, then $\pi = \pi_t$ because $\M$ is generated by the images of the $\iota_\alpha$. Thus, $\M$ together with the morphisms $\iota_\alpha: \M_\alpha \To \M$ is the coproduct of $\{\M_\alpha\}$.
\end{proof}

\begin{definition}
The coproduct of a family $\{\M_\alpha\}$ of $W^*$-algebras is their $\emph{free product}$, and we will denote it by $
\Ast_\alpha \M_\alpha$.
\end{definition}

\begin{proposition}
The category $\Wstar$ has all small colimits.
\end{proposition}

\begin{proof}
Since $\Wstar$ has all small coproducts, it's sufficient to show that $\Wstar$ has coequalizers \cite{MacLane71}. Therefore, suppose that $\rho_0, \rho_1: \M \To \N$ is a pair of morphisms, and let  $\I \subsetof \N$ be the $w^*$-closed two-sided ideal generated by $\{\rho_1(m) - \rho_0(m) \suchthat m \in \M\}$. Define $\pi: \N \To \N_=$ to be the quotient morphism by the ideal $\I$. Clearly, $\pi \circ \rho_0 = \pi \circ \rho_1$, and since the kernel of any morphism $\tilde \pi: \N \to \Tilde \N$ coequalizing $\rho_0$ and $\rho_1$ must contain $\I$, such a map must factor uniquely through $\pi$.
\end{proof}

In analogy with the category $\Set$, we call $\o(\M \oplus \N)$ the disjoint union of $\o \M$ and $\o \N$. Likewise, we call $\o(\M \ast \N)$ the collection of all possible pairs from $\o \M$ and $\o \N$. It's significant that the latter operation does not distribute over the former:

\begin{theorem}\label{TheCategory.H}
The category $\Wstar$ fails to have all coexponentials.
\end{theorem}

\begin{proof}
Suppose that $\o \Wstar$ is Cartesian closed, i.e., for any $W^*$-algebra $\N$, the functor $- \ast \N: \Wstar \To \Wstar$ has a left adjoint. All right adjoints preserve limits, so in particular for all $W^*$-algebras $\M$,  $(\M \oplus \M) \ast \N \iso (\M \ast \N) \oplus (\M \ast \N)$, but this formula is false for $\M = \CC^1$ and $\N = \CC^2$. Indeed, $\CC^1 \ast \CC^2 = \CC^2$ so the right side is commutative, but $\CC^2 \ast \CC^2$ is noncommutative since, for example, $\B(\ell^2)$ has noncommuting pairs of projections.
\end{proof}

\section{The Categorical Tensor Product}\label{Categorical}

It's easy to show that $\CC^2 \ast \CC^2$ has exactly $2^{\aleph_0}$ irreducible representations up to unitary equivalence. The collection $\o \CC^2 \ast \CC^2 $ of all pairs from $\{0,1\}$ is somewhat larger than expected! We excuse this phenomenon by  saying that although the first component of each element of $\o \CC^2 \ast \CC^2$ is indeed either $0$ or $1$, and likewise the second component, the various pairs may be structurally different. That pairs may be ``structurally different'' is both counterintuitive and desirable. For example, this interpretation permits us to say that the phase space of a quantum particle in one dimension consists of pairs of real numbers, just not pairs that can be modeled via the usual set-theoretic construction.

We interpret Theorem \ref{TheCategory.H} above as showing that we can't apply the usual definition of exponentials to obtain a collection of all functions from one collection to another, if we insist that any function should be applicable to any argument regardless of the structure of this pair. We will therefore ask that the evaluation function be defined only on pairs of a certain kind.

In this section, we focus on the collection $\o \M \tilde \tensor \N$ of pairs for which observables on the first component are compatible with the observables on the second component. The usual criterion for compatible observables is that they correspond to commuting self-adjoint operators. Therefore, in this section, we study the quotient $\M \tilde \tensor \N$ of $\M \ast \N$ by the ideal of commutators. We will later abandon this criterion.

\begin{proposition}\label{Categorical.A}
Let $\{\M_\alpha\}$ be an indexed family of $W^*$-algebras. There is a space $\M$ and a family of morphisms $\{\iota_\alpha: \M_\alpha \To \M\}$ such that $\iota_\alpha(m_\alpha)$ commutes with $\iota_\beta(m_\beta)$ whenever $m_\alpha \in \M_\alpha$, $m_\beta \in \M_\beta$, and $\alpha \neq \beta$, and if $\{\rho_\alpha: \M_\alpha \To \N\}$ is another family with this same property, then there exists a unique  morphism $\M \to \N$ such that $\rho \circ \iota_\alpha = \rho_\alpha$ for all $\alpha$.
\end{proposition}

\begin{proof}
Let $\Ast_\alpha \M_\alpha$ be the free product of the family $\{\M_\alpha\}$ with inclusions $\iota_\beta: \M_\beta \To \Ast_\alpha \M_\alpha  $, and let $\I$ be $w^*$-closed ideal generated by commutators of the form $[\iota_\alpha(m_\alpha), \iota_\beta(m_\beta)]$ for $m_\alpha \in \M_\alpha$, $m_\beta \in \M_\beta$, and $\alpha \neq \beta$. Each family $\{\rho_\alpha: \M_\alpha \To \N\}$ with the relevant property factors through $\Ast_\alpha \M_\alpha$ by the universal property of the free product, and therefore also throught $\M =(\Ast_\alpha \M_\alpha)/\I$. That it does so uniquely follows from the fact that operators of the form $\iota_\alpha(m_\alpha)$ generate $\Ast_\alpha \M_\alpha$, and therefore also $\M$.
\end{proof}

\begin{definition}
If $\{\M_\alpha\}$ is an indexed family of $W^*$-algebras, then the universal property of Proposition \ref{Categorical.A} implies that the $W^*$-algebra $\M$ is unique up to a canonical isomorphism. Following Dauns \cite{Dauns72}, we call $\M$ the categorical tensor product of the family $\{\M_\alpha\}$; we notate it $\M = \tilde \Tensor_\alpha \M_\alpha$. 
\end{definition}

\begin{lemma}\label{Categorical.C}
Let $\TT$ be the unit circle, let $\lambda$ be Lebesgue measure on $\TT$, and let $w \in \TT$ be irrational. Define $\rho_w: L^\infty(\TT,\lambda) \To L^\infty(\TT, \lambda)$ to be rotation by $w$, i.e., $(\rho_w(f))(z) = f(w\inv z)$. Then the inclusion $\iota: \CC \To L^\infty(\TT, \lambda)$ is the equalizer of $\rho_w$ with the identity morphism $\rho_1$ on $L^\infty(\TT,\lambda)$.
\end{lemma}

\begin{proof}
We show that for all projections $q\in L^\infty(\TT, \lambda)$, if $ \rho_{w}(q) =q $, then $q = 0$ or $q=1$. 

Suppose that $q$ is a projection in $L^\infty(\TT, \lambda)$ such that $\rho_{w}(q) = q$, and that $A$ is a corresponding measurable subset of $\TT$. For all intervals $I \subsetof \TT$ and naturals $n$, $\lambda(A \intersect I) = \lambda(w^nA \intersect w^n I) = \lambda(A \intersect w^nI)$. Since the powers of $w$ are dense in $\TT$, it follows that $\lambda (A \intersect I ) = \lambda (A \intersect zI )$ for all intervals $I \subsetof \TT$ and $z \in \TT$. We conclude that if $\lambda(I) = n \inv \lambda (\TT)$, then $\lambda(A \intersect I) = n \inv \lambda(A)$. It follows by the countable additivity of $\lambda$ that $\lambda(A \intersect I) = \lambda(A) \lambda (I)$.

The Lebesgue Differentiation Theorem implies that if $\lambda(A) < \lambda(\TT)$, then for every $\epsilon >0$, there is an interval $I\subsetof \TT$ such that $\lambda(A \intersect I) < \epsilon \lambda(I)$. If $q\neq 1$, then $\lambda (A) \lambda(I) = \lambda (A \intersect I)  < \epsilon \lambda (I)$ for all $\epsilon >0$, so $\lambda (A) = 0$, i.e., $q=0$.

In other words, the $W^*$-algebra $\{m \in L^\infty(\TT, \lambda)\suchthat \rho_w(m) = \rho_1(m)\}$ has $0$ and $1$ as its only projections, and therefore is equal to $\CC$. It follows that the image of any morphism $\pi: \M \To \L^\infty(\TT, \lambda)$ such that $\rho_w \circ \pi = \rho_1 \circ \pi$ is in $\CC \subset L^\infty(\TT, \lambda)$, and so factors uniquely through that inclusion.
\end{proof}

\begin{theorem}\label{Categorical.D}
Let $\TT$ be the unit circle, and let $\lambda$ be Lebesgue measure on $\TT$. The functor $- \tilde \tensor L^\infty(\TT,\lambda)  : \Wstar \to \Wstar$ does not preserve limits.
\end{theorem}

\begin{proof}
In light of Lemma \ref{Categorical.C} above, it's sufficient to show that the inclusion $\iota \tilde \tensor 1: L^\infty(\TT, \lambda) \To L^\infty(\TT, \lambda) \tilde \tensor L^\infty(\TT, \lambda)$ defined by $m \mapsto 1 \tilde \tensor m$ is not the equalizer of $\rho_1 \tilde \tensor 1,  \rho_{w} \tilde \tensor 1: L^\infty(\TT, \lambda) \tilde \tensor L^\infty(\TT, \lambda) \To L^\infty(\TT, \lambda) \tilde \tensor L^\infty(\TT, \lambda)$, where as before $w \in \TT$ is irrational, and for all $z \in \TT$, $\rho_z$ is rotation by $z$. Applying the universal property of the categorical tensor product, define $\pi: L^\infty(\TT,\lambda) \tilde \tensor L^\infty(\TT,\lambda) \To L^\infty(\TT \times \TT, \lambda \times \lambda)$ by $\pi(m_1   \tilde \tensor m_2)(z_1,z_2) = m_1(z_1) m_2(z_2)$. Similarly,  define $\delta : L^\infty(\TT,\lambda) \tilde \tensor L^\infty(\TT,\lambda) \To L^\infty(\TT,\lambda)$ by $\delta(m_1 \tilde \tensor m_2) = m_1 m_2$. 

Every nonzero projection in $L^\infty(\TT, \lambda)$ is the sum of two orthogonal projections of equal measure. Starting from the identity $1 \in L^\infty(\TT, \lambda)$, and continuing in this way, we obtain for each natural $k$ a family $\P_n$ of $2^n$ orthogonal projections of equal measure such that each family sums to the identity, and such that every projection in $\P_{n+1}$ is a subprojection of a projection in $\P_n$. We can now define a decreasing sequence of projections $q_n = \sum_{p \in \P_n}p \tilde \tensor p$, with $w^*$-limit $q \in L^\infty(\TT,\lambda) \tilde \tensor L^\infty(\TT,\lambda)$. It follows by the continuity of $\delta$, that $\delta (q) = 1$. On the other hand, $(\lambda \times \lambda)(\pi(q_n)) = 2^{-n}\lambda(\TT)^2$, so $(\lambda \times \lambda)(\pi(q)) =0$, i.e., $\pi(q)=0$. Since Lebesgue measure is rotationally invariant, we can generalize this argument to show that  $\pi((\rho_{w^k} \tilde \tensor 1)(q))=0$ for all integers $k$.

Let $\tilde p = \bigvee_k  (\rho_{w^k} \tilde \tensor 1)(q)$, and suppose that $\tilde p = 1 \tilde \tensor p$ for some projection $p \in L^\infty(\TT,\lambda)$. Because $\pi$ vanishes on each term $(\rho_{w^k} \tilde \tensor 1)(q)$, we see that $\pi(1 \tilde \tensor  p) = \pi(\tilde p) = 0$, so $p=0$. This conclusion contradicts $\delta(q)=1$, which implies that $q \neq 0$, so $\tilde p \neq 0$. Therefore, $\tilde p$ is not in the image of the inclusion $\iota \tilde \tensor 1: L^\infty(\TT, \lambda) \To L^\infty(\TT, \lambda) \tilde \tensor L^\infty(\TT, \lambda)$.

Let $\phi: \CC^2 \To L^\infty(\TT,\lambda)\tilde \tensor L^\infty(\TT,\lambda)$ be the morphism defined by $\phi(c_0, c_1) = c_0(1-\tilde p)+ c_1\tilde p$. Since by construction, $(\rho_{w} \tilde \tensor 1) (\tilde p) = \tilde p = (\rho_1 \tilde \tensor 1)(\tilde p)$, the morphism $\phi$ satisfies $(\rho_{w} \tilde \tensor 1)\circ \phi = (\rho_1\tilde  \tensor 1) \circ \phi$. However, such a morphism cannot factor through the inclusion $\iota \tilde \tensor 1$ because $\tilde p$ isn't in the image of $\iota \tilde \tensor 1$. We conclude that $\iota \tilde \tensor 1$ is not the equalizer of $\rho_1 \tilde \tensor 1,  \rho_{w} \tilde \tensor 1: L^\infty(\TT, \lambda) \tilde \tensor L^\infty(\TT, \lambda) \To L^\infty(\TT, \lambda) \tilde \tensor L^\infty(\TT, \lambda)$.
\end{proof}

\begin{corollary}
The functor $- \tilde \tensor L^\infty(\TT, \lambda): \Wstar \To \Wstar$ does not have a left adjoint.
\end{corollary}

\begin{proposition}\label{Categorical.G}
The limit in $\Wstar$ of commutative $W^*$-algebras is commutative.
\end{proposition}

\begin{proof}
Let $\NNN: \Cl \To \Wstar$ be a diagram of commutative $W^*$-algebras with limit $\M$ together with an indexed class of morphisms $\{\pi_c: \M \To \NNN(c) \suchthat c \in \Ob(\Cl)\}$. If $m_0, m_1 \in \M$, then $\pi_c(m_1 m_0 - m_1 m_0)=0$ for all objects $c$ of $\Cl$. We concluded that each morphism $\pi_c$ factors through the quotient map $\rho: \M \To \M/\I$, where $\I$ is the ideal generated by elements of the form $m_1 m_0 - m_0 m_1$ for all $m_0, m_1 \in \M$. Since, $\M$ is the limit of $\NNN$, it follows that $\rho$ is an isomorphism, so $\M$ is commutative.
\end{proof}

\begin{corollary}
The category $\cWstar$ of commutative $W^*$-algebras and unital normal $*$-homomorphisms fails to have all coexponentials.
\end{corollary}

\begin{proof}
By Proposition \ref{Categorical.G} above, a cone of commutative $W^*$-algebras is limiting in $\Wstar$ iff it is limiting in $\cWstar$. In particular, the equalizer $\iota: \CC \To  L^\infty(\TT, \lambda)$ of Lemma \ref{Categorical.C} is also an equalizer in $\cWstar$, but the inclusion $\iota \tilde \tensor 1_{L^\infty(\TT,\lambda)}$ from the proof of Theorem \ref{Categorical.D} is not an equalizer in $\cWstar$. Thus, the functor $ - \tilde \tensor L^\infty(\TT,\lambda) : \cWstar \To \cWstar$ fails to preserve limits in the category $\cWstar$, and so cannot have a left adjoint.
\end{proof}

 \section{The Spatial Tensor Product}\label{The Spatial Tensor Product}
 
We showed that the functors $- \ast \CC^2$ and $- \tilde \tensor L^\infty(\TT, \lambda)$ can't have left adjoints because they fail to preserve limits; see the proofs of Theorems \ref{TheCategory.H} and \ref{Categorical.D}. In contrast, we will show that the spatial tensor product with any given $W^*$-algebra does have a left adjoint; see Theorem \ref{Free.E}. This section defines the spatial tensor product.

Recall that we think of each quantum system $\o \M$ as a Hilbert space of quantum states modulo a group of symmetries. In quantum theory, the compound of two independent systems is represented on the tensor product of the Hilbert spaces of these constituent systems. We are thusly led to the usual tensor product of von Neumann algebras:

The spatial tensor product of von Neumann algebras $\M \subsetof \B(\H)$ and $\N \subsetof \B(\K)$ is defined to be the von Neumann algebra $\M \stensor \N  \subsetof \B(\H \tensor \K)$ generated by operators of the form $m \tensor n$, for $m \in \M$ and $n \in \N$. Remarkably, if we replace the $\M$ and $\N$ by isomorphs, then the resulting spatial tensor product will be isomorphic to the original. We thusly define the spatial tensor product of any two $W^*$-algebras.

The universal property of $\M \ast \N$ yields a monic function $\o \M \stensor \N \To  \o \M \ast \N$ that we interpret as inclusion, so that $\o \M \stensor \N$ does consist of pairs. We refer to these as \emph{Cartesian} pairs because they behave like the pairs of ordinary mathematical practice. Intuitively, two subsystems are independent if any configuration of the compound system is just a Cartesian pair of configurations, or, more practically, if a physicist would represent them on a tensor product Hilbert space.

If we wish $\stensor$ to be a functor $\Wstar \times \Wstar \To \Wstar$ in the usual sense, it's insufficient that the tensor product $\M \stensor \N$ of two $W^*$-algebras be defined up to isomorphism; for every pair of $W^*$-algebras, $\M$ and $\N$, we must choose a specific $W^*$-algebra to be called $\M \stensor \N$. This can be accomplished via the universal normal representations of $\M$ and $\N$. 

\begin{definition}
Let $- \stensor - : \Wstar \times \Wstar \To \Wstar$ be the functor defined as follows:
\begin{enumerate}
\item For every pair of $W^*$-algebras, $\M$ and $\N$, $\M \stensor \N$ is the $W^*$-subalgebra of $\B(\H \tensor \K)$ generated by operators of the form $m \tensor n$, for $m \in \M$ and $n \in \N$, where $\M \subsetof\B(\H)$ and $\N \subsetof \B(\K)$ are the universal normal representations of $\M$ and $\N$.
\item
For every pair of unital normal $*$-homomorphisms $\pi: \M_0 \to \M_1$ and $\rho: \N_0 \To \N_1$, $\pi \stensor \rho$ is defined by $ m_0 \tensor n_0 \mapsto \pi(m_0) \tensor \rho(n_0)$. Such a morphism may be constructed via a straightforward appeal to Stinespring's Theorem \ref{Normal.B}.
\end{enumerate}

\end{definition}

We could have made any other choice of representations in the above definition, and the resulting functor would've be equivalent. The naturality of such an equivalence is a coherence condition that allows us to handle morphisms between tensor product $W^*$-algebras without keeping track of which representations were used. If we wish to similarly identify $W^*$-algebras such as $(\M_0 \stensor \M_1) \stensor (\M_2 \stensor \CC)$ and $(\M_0 \stensor \M_2) \stensor \M_1$, we must establish four further coherence conditions, which together make $(\Wstar, \stensor)$ a \emph{symmetric monoidal category}.

\begin{proposition}
The category $\Wstar$ together with tensor product $\stensor$, the unit object $\CC$, and the obvious natural isomorphisms $a_{\M ,\N ,\R}:(\M \stensor \N) \stensor \R \iso \M \stensor (\N \stensor \R)$, $b_{\M,\N}:\M \stensor \N \iso \N \stensor \M$,  and $c_\M:\CC \stensor \M  \iso \M$, is a symmetric monoidal category, i.e., it satisfies the following coherence conditions:
\begin{enumerate}
\item Symmetry. $b_{\N, \M} \circ b_{\M, \N} = 1_{\M \stensor \N}$
\item Triangle Identity. $(1 \stensor c_{\N}) \circ a_{\M, \CC, \N} = (c_\M  \stensor 1_\N) \circ (b_{\M, \CC} \stensor 1_\N)$
\item Pentagon Identity. $$1_\M \stensor a _{\N, \R, \S} \circ a_{\M, \N \stensor \R, \S} \circ a_{\M, \N, \R} \stensor 1_\S = a_{\M, \N, \R \stensor \S} \circ a_{\M \stensor \N, \R ,\S}$$
\item Hexagon Identity.$$ (1_\N \stensor b_{\M, \R}) \circ a_{\N, \M, \R} \circ (b_{\M, \N} \stensor 1_\R) = a_{\N, \R, \M} \circ b_{\M, \N \stensor \R} \circ a_{\M, \N, \R} $$
\end{enumerate}
\end{proposition}

\begin{proof}
This proposition follows from the similar claim that the category of Hilbert spaces and unitary operators is a symmetric monoidal category with respect to the usual tensor product $\tensor$, and the unit Hilbert space $\ell^2_1\iso \CC$. That claim may be proved by choosing bases.
\end{proof}

\section{Bounding Cardinality}

If we plan to collect all possible functions $\o \N \To \o\M$ into a single object, we should first be sure that their number is small, i.e., is in some sense bounded by a cardinality. We do not need a tight bound, but obtain one anyway with a little extra effort.

For completeness, we begin with a proof of the following fact:

\begin{lemma}\label{Bounding.A}
If $\M \subsetof \B(\H)$ is a von Neumann algebra, then $\M$ is generated by a subset of cardinality at most $\dim \H$.
\end{lemma}

\begin{proof}
If $\dim \H$ is finite, then $\M$ is a direct sum of matrix algebras. The statement is obvious for matrix algebras, and this special case implies the general one.

If $\dim \H$ is infinite, then we may assume without loss of generality that every state on $\M$ is a vector state. The Hilbert space $\H$ has a dense subset of cardinality $\dim \H$, so the predual $\M_*$ also has a dense subset $X$ of cardinality $\dim \H$ relative to the norm topology. Let $\tau$ be the initial topology on the unit ball $\M_1\subset \M$ induced by $X$, i.e., the coarsest topology such that all the elements of $X$ are continuous. This topology is Hausdorff since if $m_0,  m_1 \in (\M_*)^* = \M$ are distinct, then they disagree on some state $\mu \in \M_*$, and therefore on some element of $X$. Since the unit ball $\M_1$ is a compact Hausdorff space in the $w^*$-topology, and $\tau$ is obviously coarser, the two topologies coincide. We conclude that the $w^*$-topology on $\M_1$ has a basis of cardinality no larger than $\dim \H$. Choosing an operator from each basis element, we obtain a dense subset of $\M_1$ of cardinality no larger than $\dim \H$; the same then follows for $\M$.
\end{proof}

If $\K$ is a Hilbert space and $\zeta \in \K$ is a vector, then $\hat \zeta: \CC \To \K$ denotes the unique linear map such that $\hat \zeta(1) = \zeta$.

\begin{lemma}\label{Bounding.B}
Let $\K$ and $\L$ be Hilbert spaces, and let  $\N \subsetof \B(\K)$ and $\M \subsetof  \B(\L) \stensor \N$ be von Neumann algebras. There is a minimum von Neumann algebra $\S \subsetof \B(\L)$ such that $\M \subsetof \S \stensor \N$, and it is generated by elements of the form $(1 \tensor \hat \zeta_0^*) m  (1 \tensor \hat \zeta_1)$ for  $\zeta_0, \zeta_1 \in \K$, and $m \in \M$ . If $\M$ has a faithful representation on a Hilbert space $\H$, then $\S$ has a faithful representation on a Hilbert space of dimension at most $2^{(\aleph_0 \cdot \dim \H \cdot \dim \K)}$.
\end{lemma}

\begin{proof}
Let $M \subsetof \M$ be a generating subset of cardinality no greater than $\dim \H$. Choose an orthonormal basis $\{\zeta_\alpha\}_{\alpha \in I}$ of $\K$.

First, we consider the case $\N = \B(\K)$. If a von Neumann algebra $\tilde \S\subsetof \B(\L)$ satisfies $M \subsetof  \tilde \S \stensor \B(\K)$, then for each generator $m \in M$, and any three basis elements $\zeta_\alpha$, $\zeta_\beta$, and $\zeta_\gamma$,
$$ [(1 \tensor \hat \zeta_\alpha^* ) m (1 \tensor \hat \zeta_\beta)] \tensor ( \hat \zeta_\gamma \hat \zeta_\gamma^*)  =(1 \tensor \hat \zeta_\gamma \hat \zeta_\alpha^*) m (1 \tensor \hat \zeta_\beta \hat \zeta_\gamma^*) \in  \tilde \S \stensor \B(\K),$$
so $ [(1 \tensor \hat \zeta_\alpha^* ) m (1 \tensor \hat \zeta_\beta)] \tensor 1 \in \tilde \S \tensor \B(\K)$, i.e., $(1 \tensor \hat \zeta_\alpha^* ) m (1 \tensor \hat \zeta_\beta) \in \tilde \S$. Thus, the von Neumann algebra $\S\subsetof \B(\L)$ generated by operators of the form $(1 \tensor \hat \zeta_\alpha^* ) m (1 \tensor \hat \zeta_\beta)$ is a subalgebra of $\tilde \S$. That $\S$ itself satisfies the inequality $M \subsetof \S \stensor \B(\K) $ follows from the following calculation for any generator $m \in M$:
$$m = \sum_{\alpha,\beta} (1 \tensor  \hat \zeta_\alpha \hat \zeta_\alpha^* ) m ( 1 \tensor \hat \zeta_\beta \zeta_\beta^*) = 
\sum_{\alpha,\beta} [( 1 \tensor \hat \zeta_\alpha^* ) m ( 1 \tensor \hat \zeta_\beta) ]  \tensor ( \hat \zeta_\alpha \hat \zeta_\beta^*)  $$
Therefore $\S \subsetof \B(\L)$ is the minimum von Neumann algebra such that $M \subsetof \B(\K) \stensor \S$, or equivalently, $\M \subsetof \B(\K) \stensor \S$.

For an arbitrary von Neumann algebra $\N \subsetof \B(\K)$, it is still true that $\S$ is the minimum von Neumann algebra such that $\M \subsetof \S \stensor \B(\K) $. By assumption $\M \subsetof \B(\L) \stensor \N$, so $\M \subsetof \S \stensor \N$. Since for any $\tilde \S$ such that $\M \subsetof  \tilde \S \stensor \N$, we trivially have $\M \subsetof \tilde \S  \stensor \B(\K)$, the von Neumann algebra $\S$ is also minimum such that $\M \subsetof \S \stensor \N$.

Our minimal von Neumann algebra $\S$ is explicitly generated by at most $ \dim \K \cdot \dim \H \cdot \dim \K =  \dim \H \cdot \dim \K$ elements, so 
by Lemma \ref{TheCategory.C}, it has a faithful representation on a Hilbert space of dimension at most $2^{\aleph_0 \cdot \dim \H \cdot \dim \K}$.
\end{proof}

The following proposition is intended to dissuade the reader from the idea that a better bound is available.

\begin{proposition}\label{Bounding.C}
If $\dim \H, \dim \K \geq 2$, then the bound given by Lemma \ref{Bounding.B} above is strict. Specifically, there is a Hilbert space $\L$, and a morphism $\pi: \B(\H) \To \B(\L) \stensor \B(\K)$ such that the minimum von Neumann algebra $\S \subsetof \B(\L)$ satisfying $\pi(\B(\H))\subsetof \S \stensor \B(\K)$ has no faithful representation on a Hilbert space of dimension less than $2^{\aleph_0 \cdot \dim \H \cdot \dim \K}$.
\end{proposition}

\begin{proof}
Fix a faithful representation of $\B(\H) \ast \B(\K)$. In particular, this is a representation of $\B(\K)$, and so is unitarily equivalent to a representation on $\L \tensor \K$ for some Hilbert space $\L$, with $\B(\K)$ acting canonically on the second factor. Let $\pi: \B(\H) \To \B(\L \tensor \K)$ be the corresponding representation of $\B(\H)$.

Let $\S\subsetof \B(\L)$ be the minimum von Neumann algebra such that $\pi(\B(\H)) \subsetof \S \stensor \B(\K)$. Clearly we also have that $\CC \stensor \B(\K) \subsetof \S \stensor \B(\K)$, so $(\pi(\B(\H)) + \CC \stensor \B(\K) )'' \subsetof \S \stensor \B(\K)$. If we now take intersection of both sides with $\B(\L) \stensor \CC$, we obtain the inequality
\begin{align*}(\pi(\B(\H)) + \CC \stensor \B(\K))'' \intersect &(\pi(\B(\H)) + \CC \stensor \B(\K))'   \\ &= (\pi(\B(\H)) + \CC \stensor \B(\K))'' \intersect \pi(\B(\H))' \intersect (\B(\L) \stensor \CC)
\\ & \subsetof
(\pi(\B(\H)) + \CC \stensor \B(\K))''  \intersect (\B(\L) \stensor \CC)
\\  &\subsetof
( \S \stensor \B(\K)) \intersect (\B(\L) \stensor \CC) 
= \S \stensor \CC.
\end{align*}
Thus the center of $\B(\H) \ast \B(\L) \iso (\pi(\B(\H)) + \CC \stensor \B(\K))''$ is a $W^*$-subalgebra $\S \stensor \CC$. The proposition then follows from the fact that $\B(\H)\ast \B(\K)$ has exactly $2^{\aleph_0 \cdot\dim\H \cdot\dim \K}$ irreducible representations up to unitary equivalence; this is Lemma \ref{Bounding.D} below.
\end{proof}

\begin{lemma}\label{Bounding.D}
Let $\H$ and $\K$ be Hilbert spaces of dimension no less then $2$. Then the $W^*$-algebra $\B (\H) \ast \B(\K)$ has exactly $2^{\aleph_0 \cdot\dim\H \cdot\dim \K}$ irreducible representations up to unitary equivalence.
\end{lemma}

\begin{proof}
We construct a set of surjective representations $\{\rho_r: \B(\H) \ast \B(\K) \To \B(\H \tensor \K)\}$, each giving $\B(\H \tensor \K)$ the structure of a ``twisted tensor product''. These representations will have different kernels.

Choose an orthonormal basis $\{\zeta_\alpha\}$ of $\K$. For every indexed family $r = \{r_\alpha: \B(\H) \To \B(\H)\}$ of automorphisms, let $\tilde r: \B(\H) \To \B(\H \tensor \K)$ be the direct sum of these representations  relative to the basis $\{\zeta_\alpha\}$ in the sense that $\tilde r(m)(\xi \tensor \zeta_\alpha) = r_\alpha(m) \xi \tensor  \zeta_\alpha$. Finally, define $\rho_r: \B(\H) \ast \B(\K) \To \B(\H \tensor \K)$ with $\B(\H)$ acting via $\tilde r$, and $\B(\K)$ acting canonically.

Each such representation $\rho_r$ is irreducible. Indeed, the image of $\rho_r$ trivially contains $\CC \stensor \B(\K)$, and because each $r_\alpha$ is an automorphism, it also contains all operators of the form $m \tensor \hat \zeta_\alpha \hat \zeta_\alpha^*$, and therefore all of $\B(\H) \tensor \CC$. Thus, $\B(\H) \tensor \B(\K) \subsetof \rho_r(\B(\H)\ast \B(\K))$.

If $r_0$ is the identity $\B(\H) \To \B(\H)$, then $r_\alpha(m_0) = m_1$ iff $(1 \tensor \hat \zeta_\alpha^*) \rho_r(m_0) (1 \tensor \hat \zeta_\alpha) = (1 \tensor \hat \zeta_0^*) \rho_r(m_1) (1 \tensor \hat \zeta_0)$, i.e., $\rho_r( \hat \zeta_0 \hat \zeta_\alpha^*) \rho_r(m_0) \rho_r(\hat \zeta_\alpha \hat \zeta_0^*) = \rho_r(\hat \zeta_0 \hat \zeta_0^*) \rho_r(m_1) \rho_r(\hat \zeta_0 \hat \zeta_0^*)$. Thus, two distinct indexed families of automorphisms $\B(\H) \To \B(\H)$ induce representations of $\B(\H) \ast \B(\K)$ with distinct kernels, if the 0-indexed automorphism of both families is the identity. Since $\B(\H)$ clearly has at least $(2^{\aleph_0})^{(\dim \H -1)}$ distinct automorphisms, we obtain $(2^{\aleph_0\cdot (\dim \H -1)})^{(\dim\K -1)} = 2^{\aleph_0 \cdot \dim\H \cdot \dim \K}$ unitarily inequivalent irreducible representations of $\B(\H) \ast \B(\K)$.

Of course, by Lemmas \ref{Bounding.A} and \ref{TheCategory.C}, the $W^*$-algebra $\B(\H) \ast \B(\K)$ cannot have more than $2^{\aleph_0 \cdot \dim \H \cdot \dim \K}$ unitarily inequivalent irreducible representations.
\end{proof}

\section{Free Exponentials}\label{Free Exponential}

We presently construct the free exponential $\M^{\ast \N}$ of $W^*$-algebras $\M$ and $\N$.  Intuitively, $\o \M^{\ast \N }$
is the collection of all functions from $\o \N $ to $\o \M$, and so $\M^{\ast \N}$ might be more appropriate called the coexponential of these algebras, but such terminology would clash with the precedent of naming objects according to their behavior in the opposite category, e.g., direct sum, tensor product, quantum group, etc. We will have constructed a functor $(-)^{\ast\N}$ which is left adjoint to the functor $- \stensor \N$ for each $W^*$-algebra $\N$, i.e., shown that the symmetric monoidal category $(\o\Wstar, \stensor)$ is closed.

\begin{theorem}\label{Free.A}
Let $\M$ and $\N$ be $W^*$-algebas. There is a $W^*$-algebra $\M^{\ast \N}$ and a morphism $\varepsilon: \M \To \M^{\ast \N}\stensor \N$, that is initial among such pairs in the sense that if $\R$ is another $W^*$-algebra and $\pi: \M \To \R \stensor \N$ a morphism, then there exists a unique morphism $\rho: \M^{\ast \N} \To \R$ such that $\pi = (\rho \stensor 1) \circ \varepsilon$.
\end{theorem}

\begin{proof}
We may assume that $\M \subsetof \B(\H)$ and $\N \subsetof \B(\K)$  are von Neumann algebras. We now repeatedly apply Lemma \ref{Bounding.B}.

 Let $\kappa =2^{\aleph_0 \cdot \dim \H \cdot \dim \K}$, and write $\ell^2_\kappa = \ell^2(\kappa)$. Define $\varepsilon: \M \To \bigoplus_\sigma \S_\sigma \stensor \N$ by $\varepsilon(m) = \bigoplus_\sigma \sigma(m)$, where both sums are taken over the set of morphisms $\sigma: \M \To  \B(\ell^2_\kappa) \stensor \N$, and  $\S_\sigma \subsetof \B(\ell^2_\kappa)$ is the minimum von Neumann algebra such that $\sigma(\M)\subsetof  \S_\sigma \stensor \N$. Clearly $\varepsilon (\M) \subsetof \bigoplus_\sigma (\S_\sigma \stensor \N) = (\bigoplus_\sigma \S_\sigma) \stensor \N$, so let  $\M^{\ast \N} \subsetof \bigoplus_\sigma \S_\sigma$ be the minimum von Neumann algebra such that $\varepsilon(\M) \subsetof  \M^{\ast \N} \stensor \N$.

Suppose that $\R$ is a $W^*$-algebra, and $\pi: \M \To \R \stensor \N$ a morphism. By Lemma \ref{Bounding.B}, there is a subalgebra $\S \subsetof \R$ that has a faithful representation on a Hilbert space of dimension $\kappa$ and satisfies $\pi(\M) \subsetof \S \stensor \N$. Choosing such a faithful representation, we write $\S \subsetof \B(\ell^2_\kappa)$. By construction there is a morphism $\rho: \M^{\ast\N} \To \S$ such that $\pi = (\rho \stensor 1) \circ \varepsilon $. After including $\S \subsetof \R$, we obtain a morphism  $\rho_0: \M^{\ast\N} \To \R$, which satisfies this same equation.

If $\rho_1 : \M^{\ast \N} \To \R$ is another morphism satisfying this equation, then for all $\zeta_0, \zeta_1 \in \K$, and $m \in \M$, 
$$\rho_i((1 \tensor \hat\zeta_0^* ) \varepsilon(m)(1 \tensor \hat \zeta_1 )) = (1 \tensor \hat\zeta_0^* )( \rho_i \stensor 1) (\varepsilon(m))(1 \tensor \hat \zeta_1) =  (1 \tensor \hat\zeta_0^* )\pi(m)(1 \tensor \hat \zeta_1),$$
for each $i= 0, 1$. Since elements of the form $(1 \tensor \hat\zeta_0^* ) \varepsilon(m)(1 \tensor \hat \zeta_1)$ generate $\M^{\ast \N}$, we conclude that $\rho_1 = \rho_0$, as desired.
\end{proof}

Because the pair $(\M^{\ast \N}, \varepsilon)$ constructed in Theorem \ref{Free.A} above satisfies a universal property, we may in the usual manner show that it is unique up to a canonical isomorphism of $\M^{\ast \N}$.  We may thus make the following definition:

\begin{definition}
Let $\M$ and $\N$ be $W^*$-algebras. The \emph{free exponential} of $\M$ and $\N$ is defined up to isomorphism, as a $W^*$-algebra $\M^{\ast \N}$ together with a morphism $\varepsilon_{\M, \N}: \M \To  \M^{\ast \N} \stensor \N$ satisfying the universal property of Theorem \ref{Free.A}. 
\end{definition}

We view the collection $\o\M^{\ast \N} \stensor \N$ as the Cartesian product of $\o \M^{\ast \N}$ and $\o \N$, and the morphism $ \o \varepsilon: \o \M^{\ast \N} \stensor \N \To  \o\M$ as evaluation of the first argument on the second.

\begin{proposition}
If $\M\subsetof \B(\H)$ and $\N\subsetof\B(\K)$ are von Neumann algebras, then their free exponential $\M^{\ast \N}$ can be faithfully represented on a Hilbert space of dimension no greater than $2^{\aleph_0 \cdot \dim \H \cdot \dim \K}$. Conversely, if $\H$ and $\K$ are Hilbert spaces of dimension greater than $1$, then every faithful representation of $\B(\H)^{\ast \B(\K)}$ is on a Hilbert space of dimension no less than $2^{\aleph_0 \cdot \dim \H \cdot \dim \K}$.
\end{proposition}

\begin{proof}
Clearly, $\M^{\ast \N}$ can have no proper $W^*$-subalgebras $\R \subset \M^{\ast \N}$ such that $\varepsilon(\M) \subsetof \R \stensor \N$. Apply Lemma \ref{Bounding.B}. The partial converse follows from Proposition \ref{Bounding.C}.
\end{proof}

\begin{definition}
For each $W^*$-algebra $\N$, let $(-)^{\ast \N}: \Wstar \To \Wstar$ be the functor defined as follows:
\begin{enumerate}
\item For every $W^*$-algebra $\M$, let $\M^{\ast \N}$ be defined by the construction in the proof of Theorem \ref{Free.A}, taking $\M \subsetof \B(\H)$ and $\N \subsetof \B(\K)$ to be the universal normal representations of these algebras.
\item For every morphism $\pi: \M_0 \To \M_1$, let $\pi^{\ast \N}: \M_0^{\ast\N} \To \M_1^{\ast \N}$ be the unique morphism such that $( \pi^{\ast \N} \stensor 1_\N) \circ \varepsilon_0 = \varepsilon_1 \circ \pi$, where $\varepsilon_i: \M_i \mapsto \M_i^{\ast \N} \stensor \N$, for $i = 0,1$, is the morphism constructed in the proof of Theorem \ref{Free.A}.
\end{enumerate}
\end{definition}

\begin{theorem}\label{Free.E}
For every $W^*$-algebra $\N$, the functor $(-)^{\ast \N}$ is left adjoint to the functor $- \stensor \N$. Thus, the monoidal category $(\o\Wstar, \stensor)$ is closed.
\end{theorem}

\begin{proof}
The universal property of $\varepsilon: \M \To  \M^{\ast \N} \stensor \N$ can be rephrased to say that for every $W^*$-algebra $\R$, the function $f_{\M, \R}: \Hom(\M^{\ast \N}, \R) \To \Hom(\M, \R \stensor \N)$ defined by $\rho \mapsto (\rho \stensor 1)\circ \varepsilon$ is a bijection. It therefore remains to show that this function is a natural transformation of functors $\o \Wstar \times \Wstar \To \Set$, i.e., that for any pair of morphisms, $\pi: \M_1 \To \M_0$ and $\phi: \R_0 \To \R_1$, we have that $\Hom(\pi, \phi \stensor \N) \circ f_{\M_0, \R_0} = f_{\M_1, \R_1} \circ \Hom(\pi^{\ast \N}, \phi)$ as functions $\Hom(\M_0^{\ast \N}, \R_0) \To \Hom(\M_1,  \R_1 \stensor \N)$. This follows from the following calculation, valid for any $\rho_0 \in \Hom(\M_0^{\ast \N}, \R_0)$:
\begin{align*}
[\Hom(\pi, \phi \stensor \N) \circ f_{\M_0, \R_0}](\rho_0)&  = \Hom(\pi, \phi\stensor \N) (( \rho_0\stensor 1)\circ \varepsilon_0) \\& = (\phi \stensor 1) \circ  (\rho_0\stensor1) \circ \varepsilon_0 \circ \pi   = ((\phi  \circ \rho_0\circ \pi^{\ast \N}) \stensor 1) \circ \varepsilon_1 \\& = [f_{\M_1,\R_1} \circ \Hom(\pi^{\ast\N}, \phi)](\rho_0).
\end{align*}
 \end{proof}
 
The results of the rest of this section are all more or less corollaries of Theorem \ref{Free.E} above, so the reader who is familiar with closed symmetric monoidal categories will find little new here. Therefore, what follows is intended primarily as a discussion of this setting.
 
Recall that left adjoints preserve colimits and right adjoints preserve limits. Thus, we obtain our first corollary:

\begin{corollary}\label{Free.F}
The functor $(-)^{\ast \N}$ preserves colimits. The functor $- \stensor \N $ preserves limits.
\end{corollary}

In particular, this means that for any indexed family $\{\M_\alpha\}$ of $W^*$-algebras, $(\Ast_\alpha \M_\alpha)^{\ast \N} \iso \Ast_\alpha \M_\alpha^{\ast \N}$, and $\N \stensor (\bigoplus_\alpha \M_\alpha) \iso \bigoplus_\alpha \N \stensor \M_\alpha$. While the latter claim won't surprise anyone, the former is sufficiently reassuring that we'll state it formally:

\begin{corollary}\label{Free.G}
Let $\{\M_\alpha\}$ be an indexed family of $W^*$-algebras. Then $(\Ast_\alpha \M_\alpha)^{\ast \N} \iso \Ast_\alpha \M_\alpha^{\ast \N}$.
\end{corollary}

What about the other familiar laws of exponentiation? We might for example expect that $(\M^{\ast \N})^{\ast\R} \iso\M^{\ast (\N \ast \R)}$. A small amount of abstract nonsense involving the universal property of the free exponential instead reveals the following:

\begin{corollary}
$(\M^{\ast \N})^{\ast\R} \iso \M^{\ast (\R \stensor \N)}$
\end{corollary}

The final such property follows from Yoneda's lemma. For fixed $W^*$-algebras $\M$, $\N_1$ and $\N_2$, we obtain the following familiar sequence of natural transformations:
\begin{align*}\Hom(\M^{\ast \N_1  \oplus \N_2}&, \R)  \iso \Hom(\M, \R \stensor (\N_1\oplus \N_2)) \iso \Hom (\M, (\R \stensor \N_1) \oplus (\R \stensor \N_2)) \\ & \iso \Hom(\M,  \R \stensor \N_1) \times \Hom(\M, \R \stensor \N_2) \\& \iso \Hom(\M^{\ast \N_1}, \R) \times \Hom(\M^{\ast \N_2}, \R) \iso \Hom(\M^{\ast\N_1}\ast \M^{\ast\N_2}, \R)
\end{align*}
The natural transformations that got us to and from the Cartesian product of Hom sets arise from the the universal properties of $\oplus$ and $\ast$ respectively. The same proof works for any collection of $W^*$-algebras.

\begin{corollary}\label{Free.I}
Let $\{\N_\alpha\}$ be an indexed family of $W^*$-algebras. Then $\M^{\ast(\bigoplus_\alpha \N_\alpha)} \iso \Ast_\alpha \M^{\ast \N_\alpha}$.
\end{corollary}

We now have an opportunity to justify the notation $\M^{\ast \N}$, by combining the above corollary with the simple observation that $\M^{\ast \CC} \iso \M$.

\begin{corollary}
Let $X$ be a set. Then $\M^{\ast\ell^\infty(X)} \iso \Ast_{x \in X} \M$.
\end{corollary}

Using the universal property of free exponentials, we can define a contravariant functor $\M^{\ast (-)}$. By Yoneda's lemma, it is the unique functor such that the canonical bijection $\Hom(\M^{\ast \N}, \R) \iso \Hom(\M, \N \stensor \R)$ is natural in $\N$, as well as  in $\M$ and $\R$. The calculation $\Hom(\M^{\ast \N}, \R) \iso \Hom(\M , \N \stensor \R) \iso \Hom(\M, \R \stensor \N) \iso \Hom(\M^{\ast \R}, \N)$ shows that curiously the functor $\M^{\ast (-)}$ is its own adjoint (it's contravariant). Therefore it sends all limits  to colimits.

\begin{corollary}
The assignment $\N \mapsto \M^{\ast \N}$ extends uniquely to a contravariant functor such that the canonical bijection $\Hom(\M^{\ast \N}, \R) \iso \Hom(\M, \R \stensor \N)$ is natural in $\N$. It sends limits to the corresponding colimits.
\end{corollary}

Note that the isomorphisms of Corollaries \ref{Free.G}-\ref{Free.I} are natural in every variable.

We wish to draw an analogy between the functor $\Wstar \times \o\Wstar \To \o\Wstar$ given by $(\M, \o\N) \mapsto \o\M^{\ast \N}$, and the $\Hom$ functor $ \Wstar \times \o\Wstar \To \Set$. The adjunction $\Hom (\M^{\ast\N}, \CC)\iso \Hom(\M, \N \stensor \CC) \iso \Hom(\M, \N)$ establishes a canonical bijection between the atoms of $\o\M^{\ast \N}$ and the functions from $\o \N$ to  $\o\M$. Thus, $\o \M^{\ast \N}$ consists of $\Hom(\o \N,\o \M)$ together with other functions that cannot be individually distinguished. We formally state this correspondence.

\begin{corollary}
$\Hom (\M^{\ast\N}, \CC)\iso \Hom(\M, \N)$ 
\end{corollary} 

Let us write $\gamma_\pi: \M^{\ast\N} \To \CC$ for the character corresponding to a morphism $\pi: \N \To \M$. In particular, for every $W^*$-algebra $\M$, we have the character $\gamma_\M = \gamma_{1_\M}: \M^{\ast \M} \To \CC$ corresponding to the identity morphism. It is the counit of the canonical cocomposition operation defined as follows:

\begin{definition}
For each fixed $W^*$-algebra $\N$, \emph{cocomposition} is the natural transformation
$$\kappa_{\M, \N, \R}: \M^{\ast \R} \To \M^{\ast \N} \stensor \N^{\ast \R}$$ defined to be the unique map such that $(\kappa_{\M, \N, \R} \stensor 1) \circ \varepsilon_{\M, \R} = (1 \stensor \varepsilon_{\N, \R}) \circ \varepsilon_{\M, \N}$.
\end{definition}

We claim the obvious desirable properties:

\begin{corollary}
Let $\M$, $\N$, $\R$, and $\S$ be $W^*$-algebras.
\begin{enumerate}
\item
Cocomposition is coassociative in the sense that
$$(\kappa_{\M, \N, \R} \stensor 1) \circ \kappa_{\M, \R, \S} = (1 \stensor \kappa_{\N, \R, \S}) \circ \kappa_{\M, \N, \S}.$$
Both are morphisms $\M^{\ast \S} \To \M^{\ast \N} \stensor \N^{\ast \R} \stensor \R^{\ast \S} $.

\item
The identity characters are the counits of cocomposition in the sense that the morphisms 
$ (\gamma _\M \stensor 1) \circ \kappa_{\M, \M, \N}$ and $(1 \stensor \gamma_\N) \circ \kappa_{\M, \N, \N}$ are both equal to the identity morphism $1: \M^{\ast \N} \To \M^{\ast \N}$.
\end{enumerate}
\end{corollary}

\begin{proof}
Let's prove the second claim first. We calculate:
\begin{align*}
&[((\gamma_\M \stensor 1) \circ \kappa_{\M, \M, \N})\stensor 1] \circ \varepsilon_{\M,\N}
= (\gamma_\M \stensor 1 \stensor 1) \circ (\kappa_{\M, \M, \N} \stensor 1) \circ \varepsilon_{\M, \N}
\\ &= (\gamma_\M \stensor 1\stensor 1) \circ (1 \stensor \varepsilon_{\M, \N}) \circ \varepsilon_{\M, \M}
= \varepsilon_{\M, \N} \circ (\gamma_\M \stensor 1) \circ \varepsilon_{\M, \M}
= \varepsilon_{\M, \N} \circ 1_\M = \varepsilon_{\M,\N}
\end{align*}
The universal property of $\varepsilon_{\M, \N}: \M \To \M^{\ast \N} \stensor \N$ then yields that $(\gamma _\M \stensor 1) \circ \kappa_{\M, \M, \N}=1$. The equality $(1 \stensor \gamma_\N) \circ \kappa_{\M, \N, \N} = 1$ is proved analogously.

For the second claim, a similar calculation to the one above yields the equality
\begin{align*}[((\kappa_{\M, \N, \R} \stensor 1) \circ \kappa_{\M, \R, \S})\stensor 1] \circ \varepsilon_{\M, \S} & = (1 \stensor (1 \stensor \varepsilon_{\R,\S}) \circ \varepsilon_{\N, \R}))\circ \varepsilon_{\M,\N} \\ &= [((1 \stensor \kappa_{\N, \R, \S}) \circ \kappa_{\M, \N, \S}) \stensor 1] \circ \varepsilon_{\M, \S}.\end{align*}
Again, by the universal property of $\varepsilon_{\M,\S}: \M \To \M^{\ast \S} \stensor \S$, we conclude that $(\kappa_{\M, \N, \R} \stensor 1) \circ \kappa_{\M, \R, \S} = (1 \stensor \kappa_{\N, \R, \S}) \circ \kappa_{\M, \N, \S}$.
\end{proof}

Thus, the functor $(\M, \o\N) \mapsto \o\M^{\ast \N}$ really does behave very much like $\Hom: \Wstar \times \o\Wstar \To \Set$. If we replace the latter with the former, the category $\o\Wstar$ becomes an \emph{enriched category}, enriched over itself. The term `closed' in Theorem \ref{Free.E} refers to the fact that $(\o \Wstar, \stensor)$ is canonically equipped with such an enrichment, and is thusly ``closed under the formation of hom objects''.

If we look at the hom objects of our enriched category $\o \Wstar$ through the functor $\Hom(\o\CC, -)$, we recover the original category $\o\Wstar$. Indeed, the natural bijective correspondence $ \Hom(\o\CC, \o\M^{\ast \N}) \iso \Hom(\o\N, \o\M)$ respects composition in the sense that if $\pi: \M \To \N$ and $\rho: \N \To \R$ are morphisms, then $(\gamma_\pi \stensor \gamma_\rho) \circ \kappa_{\M, \N, \R} = \gamma_{\rho \circ \pi}$. We deduce this equality from the the following calculation:
\begin{align*}
(((&\gamma_\pi \stensor \gamma_\rho) \circ \kappa_{\M, \N, \R}) \stensor 1) \circ \varepsilon_{\M, \R}
=
(\gamma_\pi \stensor \gamma_\rho \stensor 1) \circ (\kappa_{\M, \N, \R} \stensor 1) \circ \varepsilon_{\M, \R}
\\ &=
(\gamma_\pi \stensor \gamma_\rho \stensor 1) \circ (1 \stensor \varepsilon_{\N, \R}) \circ \varepsilon_{\M, \N}
=
(\gamma_\rho \stensor 1) \circ \varepsilon_{\N, \R} \circ (\gamma_\pi \stensor 1) \circ \varepsilon_{\M, \N}
=
\rho \circ \pi
\end{align*}

\begin{corollary}\label{Free.O}
Let $\pi: \M \To \N$ and $\rho: \N \To \R$ be morphisms of $W^*$-algebras. Then, $(\gamma_\pi \stensor \gamma_\rho) \circ \kappa_{\M, \N, \R} = \gamma_{\rho \circ \pi}$.
\end{corollary}

The present paper embraces the viewpoint that the enriched category $\o\Wstar$ is the real object of our study, whereas the original category $\o\Wstar $ is simply a view of this entity from the category $\Set$, in which mathematics is ordinarily developed. Thus, the collection $\o \M^{\ast \N}$ consists of all possible functions from $\o\N$ to $\o\M$, whereas the set $\Hom(\o\N, \o\M)$ consists just of those that are visible from $\Set$.

The enrichment of $\o\Wstar$ does not cut ties with $\Set$ completely: Although we have replaced Hom sets with hom objects, composition is still a morphism in the ordinary sense. However, we may preserve the sense that $\o\Wstar$ is a self-contained entity via the observation that $\Set$ behaves like a subcategory of enriched $\o\Wstar$: Each set is a quantum collection via the identification $X \mapsto \o \ell^\infty(X)$, and the set $X^Y = \o \ell^\infty(\Hom(Y, X))$ is just the maximum subset of the collection $\o \ell^\infty(X)^{\ast \ell^\infty(Y)}$.

We conclude this section by examining Corollaries \ref{Free.G}-\ref{Free.I} in the context of the enriched category $\o\Wstar$. Recall that we interpret $\o\N_0 \ast \o\N_1$ to be the collection all pairs from $\o \N_0$ and $\o \N_1$. Thus, the equation $\M^{\ast \N_0} \ast \M^{\ast \N_1} \iso \M^{\ast (\N_0 \oplus \N_1)}$ says that $ \o\N_0 \oplus \N_1$ is the coproduct of $\o\N_0$ and $\o\N_1$ within the enriched category. Similarly, the equation $\M_0^{\ast \N} \ast \M_1^{\ast \N} \iso (\M_0 \ast \M_1)^{\ast \N}$ says that $\o \M_0 \ast \M_1$ is the product $\o\M_0$ and $\o\M_1$ within the enriched category. Finally, the natural isomorphism $(\M^{\ast \N})^{\ast \R} \iso \M^{\ast (\R \stensor \N)}$ is the adjunction between functors $\o(-)^{\ast \N}$ and $\o -\stensor \N$ within the enriched category.

\section{Quantum Operations}\label{Quantum Operations}

Each physical operation between two quantum systems is given by a function $f: S_\N \To S_\M$, where $\N$ and $\M$ are the $W^*$-algebras associated to the source and target systems respectively. Since the affine combination of states corresponds to probabilistic mixing, the function $f$ should respect the affine structure of $S_\N$. Thus, $f$ extends uniquely to a bounded $\CC$-linear function $\N_* \To \M_*$, and so we obtain a normal linear function $f^*: \M \To \N$. If $m\in \M$ is positive, then for all states $\nu \in S_\N$, $\nu(f^*(m)) =(f(\nu))(m) \geq 0$, so $f^*(m)$ is positive, and more generally $f^*$ is positive. A similar computation leads us to conclude that $f^*$ must be a unital normal positive map.

We assume that any quantum operation may be implemented by a physical process that leaves any number $n$ of independent qubits unaffected. Thus, we begin anew with a function $f_n: S_{\N\stensor M_{2^n}(\CC)} \To S_{\M  \stensor M_{2^n}(\CC)}$ such that $f_n(\nu \stensor \varphi) = f(\nu) \stensor \varphi$ for all states $\nu \in S_\N$ and $\varphi \in S_{M_{2^n}(\CC)}$. Repeating the above argument, we obtain a family of unital normal positive maps $f_n^*: \N \stensor M_{2^n}(\CC) \To \M \stensor M_{2^n}(\CC)$ such that $f_n^*(m \tensor x) = f^*(m) \tensor x$ for all $m \in \M$ and $x \in M_{2^n}(\CC)$. The existence of such a family is a nontrivial condition on $f^*$ called \emph{completely positivity}.

We've essentially condensed Kraus' original argument \cite{Kraus83} for the formalization of quantum operations as completely positive maps, generalizing it to $W^*$-algebras.

\begin{definition}
A linear map $\psi: \M \To \N$ is \emph{completely positive} in case whenever $[m_{ij}]$ is a positive matrix with entries $m_{ij} \in \M$, then the matrix $[\psi(m_{ij})]$ is also positive.
\end{definition}

One can show that a normal positive map $\psi: \M \To \N$ is completely positive iff for all $W^*$-algebras $\R$, there is a normal positive map $(\psi \stensor 1): \M \stensor \R \To \N \stensor \R$ defined by $(\psi \stensor 1)(m \tensor r) = \psi(m) \tensor r$. Evidently, every normal $*$-homomorphism is completely positive.

Our focus on the class of unital normal completely positive maps can also be motivated via Stinespring's Theorem, of which there is a proof in the appendix. This fundamental result is essentially an adaptation of the GNS construction to completely positive maps:

\begin{theorem}[Stinespring's Theorem]
Let $\M$ be a $W^*$-algebra, and $\K$ a Hilbert space. If $\psi: \M \To \B(\K)$ is  a normal completely positive map, then there exist a Hilbert space $\H$, a normal unital $*$-homomorphism $\sigma: \M \To \B(\H)$, and a bounded operator $v \in \B(\K, \H)$, such that $\psi(m) = v^*\sigma(m) v$ for all $m \in \M$.
\end{theorem}

A straightforward consequence of this theorem is that a function $\psi: \M \To \N$ is a normal unital completely positive map iff for some pair of faithful representations of $\M \subsetof \B(\H)$ and $\N \subsetof \B(\K)$ as von Neumann algebras, there exists an isometry $v: \K \To \H$ such that $\psi(m) = v^* m v$. Thus, the laundry list of properties that defines unital normal completely positive maps turns out to be quite cohesive.

\begin{definition}
Let $\pWstar$ be the category
\begin{enumerate}
\item whose objects are $W^*$-algebras, and
\item whose morphisms are unital normal completely positive maps.
\end{enumerate}
Repeating the argument of Section \ref{The Spatial Tensor Product}, we see that the spatial tensor product $\stensor$ makes $\pWstar$ into a symmetric monoidal category.
\end{definition}

Thus, $\o\pWstar$ has the same objects as $\o\Wstar$, but more morphisms.  A morphism $\o \psi: 
\o \CC \To \o\M$ in $\o\pWstar$ is just a state on $\M$. States are typically considered to be the noncommutative analogues of probability measures, and indeed the states on $L^\infty(X, \mu)$ are in canonical bijective correspondence with the probability measures that are absolutely continuous with respect to $\mu$. Because to us $\o\M$ is a collection, we think of $\o \psi$ as a probability \emph{distribution}.

If $X$ and $Y$ are sets, then a morphism $\o\psi: X \To Y$ in $\o \Wstar$ is just an assignment of probability measures on $Y$ to the elements of $X$. Thus, we will think of the morphisms of $\o \Wstar$ as \emph{probabilistic} quantum operations; intuitively, each initial configuration is transformed randomly to a final configuration.

Given this motivation, it's natural to ask whether the unital normal completely positive maps between two given $W^*$-algebras arise somehow from the unital normal $*$-homomorphisms between them, i.e., whether probabilistic quantum operations are mixtures of deterministic ones. Interpreted naively, the answer to this question is no because sometimes there are no unital normal $*$-homomorphisms between two $W^*$-algebras. For example, there are no unital normal $*$-homomorphisms between $M_2(\CC)$ and $L^\infty(\RR, \lambda)$ in either direction.

We rescue this hypothesis by showing that every probabilistic quantum operation $\o \N \To \o \M$ is indeed a mixture of deterministic quantum operations $\o \N \To \o \M$ in the sense that it arises from a probability distribution on the collection $\o \M^{\ast \N}$.

\section{Completely Positive Maps from States on the Free Exponential}

Fix two $W^*$-algebras $\M$ and $\N$, and let $\varepsilon: \M \To \M^{\ast \N} \stensor \N$ be the coevaluation morphism, as in the Section \ref{Free Exponential}. We saw that there is a natural bijection $\Hom(\M^{\ast \N}, \CC) \iso \Hom(\M, \N)$ given by $\gamma \mapsto (\gamma \stensor 1 ) \circ \varepsilon$. If we look at this picture in the opposite category $\o \Wstar$, we see that we begin with a configuration of $\o\N$, then prepare a configuration of $\o\M^{\ast \N}$, and finally evaluate. What happens if we prepare the configuration of $\o \M^{\ast \N}$ probabilistically, i.e., replace $\gamma$ with a non-homomorphic state?

We saw in Section \ref{Quantum Operations} above that the class of $W^*$-algebras, together with their unital normal completely positive maps and the spatial tensor product, forms a symmetric monoidal category. Thus, for any state $\mu: \M^{\ast \N} \To \CC$, the expression $(\mu \stensor 1) \circ \varepsilon$ defines a unital normal completely positive map $\M \To \N$. So, indeed, if we select a function from $\o\N$ to  $\o\M$ randomly, then we obtain what we've called probabilistic function between these collections. The goal of this section is to show that every morphism of $\pWstar$ can be obtained in this way. I hope that this is seen as further evidence for the soft thesis that a unital completely positive map is the same thing a probabilistic function in the opposite direction.

\begin{theorem}\label{Completely.A}
The formula $\mu \mapsto (\mu \stensor 1) \circ \varepsilon_{\M,\N}$ defines a \emph{surjective} natural transformation $\Hom_\pWstar(\M^{\ast \N}, \CC) \To \Hom_\pWstar(\M, \N)$.
\end{theorem}

The demonstration of naturality is simply a matter of writing out the relevant definitions. The real content of the above theorem is in the claim of surjectivity. We prove the theorem in two steps, each a lemma.

\begin{lemma}\label{Completely.B}
Let $\M\subsetof \B(\H)$ and $\N\subsetof \B(\K)$ be von Neumann algebras, and let $\pi: \M \To \N$ be an isomorphism. If $\kappa\geq\dim(\H), \dim(\K)$ is an infinite cardinal, then there is a unitary operator $u \in \B(\K \tensor \ell^2_\kappa, \H \tensor \ell^2_\kappa)$ such that for all $m \in \M$, $(m \tensor 1)u = u(\pi(m) \tensor 1)$. Furthermore, if $v \in \B(\K, \H)$ is any given partial isometry such that $ m v = v\pi(m)$, then we can construct $u$ so that $(1 \stensor \hat e_0^*) u(1 \stensor \hat e_0) $ extends $v$, where $\{e_\alpha\}_{\alpha < \kappa}$ is the standard basis of $\ell^2_\kappa$.
\end{lemma}

\begin{proof}
The proof is a back-and-forth construction based on the fact that any two cyclic representations of $\M$ are unitarily equivalent iff the states associated to their cyclic vectors are equal. This is the only essential content of the proof whose details are spelled out below, and the reader is encouraged to skip it.

We may assume that every normal state of $\M$ is a vector state, and likewise for $\N$, by replacing $\H$ and $\K$ by $\H \tensor \ell^2$ and $\K \tensor \ell^2$ respectively. Choose bases for both $\H$ and $\K$, and so obtain well-ordered bases of $\H \tensor \ell^2_\kappa$ and $\K \tensor \ell^2_\kappa$ of order type $\kappa$, consisting of vectors of the form $\xi \tensor e_\alpha$. 

We construct the desired unitary operator by transfinite induction. At each stage $\gamma \leq \kappa$ of the construction, we have a partial isometry $u \in \B(\K \tensor \ell^2_\kappa, \H \tensor \ell^2_\kappa)$ satisfying the following properties:
\begin{enumerate}
   \item For all $m \in \M$,  $(m \tensor 1) u = u (\pi(m) \tensor 1)$.
   \item The domain projection $u^*u$ can be written as a disjoint sum $u^*u = \sum_{\alpha \leq \gamma} p_\alpha$ with $p_\alpha \leq 1 \tensor \hat e_\alpha \hat e_\alpha^*$. Likewise for the range projection $uu^*$.
\end{enumerate}

At stage zero, set $u = (1 \tensor \hat e_0) v(1 \tensor \hat e_0^*)$. At each odd successor stage, find the least basis vector $\zeta \tensor e_\alpha$ not in $u^*u(\K \tensor \ell^2_\kappa)$, and let $\zeta^{\perp} \tensor e_\alpha$ be the normal component of $\zeta \tensor e_\alpha$ to $u^*u(\K \tensor \ell^2_\kappa)$. Find the least basis vector $e_\beta$ such that $uu^* (1 \tensor \hat e_\beta \hat e_\beta^*) = 0$. Choose a vector $\xi \in \H$ such that $\langle \xi | \cdot \xi\rangle = \|\zeta^{\perp}\|^{-2}\langle\zeta^{\perp}| \pi(\cdot) \zeta^\perp \rangle$. By Theorem 3.3.7 in Pedersen's \emph{$C^*$-algebras and Their Automorphism Groups} \cite{Pedersen79}, the identity representation of $\M$ on $\M\xi$ is unitarily equivalent to the representation $\pi$ on $\N \zeta^\perp$, so we can extend $u$ to a partial isometry that sends $\|\zeta^\perp\|\inv \zeta^\perp \tensor e_\alpha$ to $\xi \tensor e_\beta$, and satisfies the properties enumerated above. At each even successor stage, we instead begin with the least basis vector $\xi \tensor e_\alpha$ not in $uu^*(\H \tensor \ell^2_\kappa)$, and proceed in the opposite direction. At each limit stage, we take the ultraweak limit of the partial isometries constructed at previous stages, which are all compatible.

After stage $2\alpha$ of the construction, the $\alpha^{th}$ basis vector of $\K\tensor \ell^2_\kappa$ is in $u^*u(\K\tensor \ell^2_\kappa)$ and the $\alpha^{th}$ basis vector of $\H \tensor \ell^2_\kappa$ is in $uu^*(\H \tensor \ell^2_\kappa)$. Thus, at the final stage $\kappa$, $u^*u=1$ and $uu^*=1$, so the desired unitary operator has been constructed.
\end{proof}

\begin{lemma}\label{Completely.C}
Let $\psi: \M \To \N$ be a unital normal completely positive map of $W^*$-algebras. There exists a Hilbert space $\L$, a morphism $\pi: \M \To \B(\L) \stensor \N$, and a state $\mu: \B(\H) \To \CC$, such that $(\mu \stensor 1) \circ \pi = \psi$.
\end{lemma}

\begin{proof}
We may assume without loss of generality that $\N \subsetof \B(\K)$ is a von Neumann algebra. Applying Stinespring's Theorem to $\psi$, we find a representation $\sigma: \M \To \B(\H)$ and an isometry $v \in \B(\K , \H)$ such that $\psi = v^*\sigma(\cdot )v$ and $\H = \overline{\sigma(\M)v\K}$.

We now show that $\N'$ acts on $\H$ via $\phi(n'): \sigma(m) v \zeta \mapsto \sigma(m) v (n'\zeta)$. Fix $n' \in \N' \subsetof \B(\K)$, and let $\sum_i \sigma(m_i) v \zeta_i$ be a finite sum. We are interested in showing that the following quantity is positive:
\begin{align*}
\| n'\|^2 \| \sum_i \sigma(m_i) & v \zeta_i\|^2 - \| \sum_i \sigma(m_i) v (n' \zeta_i)\|^2
\\ &=
\| n' \|^2 \sum_{i,j} \langle \zeta_i| v^* \sigma(m_i^*m_j) v \zeta_j\rangle - \sum_{i,j} \langle \zeta_i | n'^* v^* \sigma(m_i^* m_j) v n' \zeta_j \rangle
\\ &=
\sum_{i,j} \left\langle \zeta_i \middle| \left( \|n'\|^2 \psi(m_i^* m_j) - n'^* \psi(m_i^*m_j)n' \right) \zeta_j\right\rangle 
\end{align*}
Thus, it's sufficient to show that the matrix $[\|n'\|^2 \psi(m_i^* m_j) - n'^* \psi(m_i^*m_j)n' ]_{i,j}$ is positive. Noting that $\psi(m_i^* m_j) \in \N$, and therefore commutes with $n'$ and $n'^*$, we calculate:
\begin{align*}
[\|n'\|^2 \psi(m_i^*m_j) - n'^*\psi(m_i^*m_j) n']
& =
[(\|n'\|^2 - n'^* n') \psi(m_i^* m_j)]
\\ &=
(\|n'\|^2 - n'^* n')^{1/2}  [\psi(m_i^* m_j) ] (\|n'\|^2 - n'^* n')^{1/2}
\end{align*}
Since $\psi$ is completely positive, this yields the desired conclusion. Thus we have defined a bounded operator $\phi(n') \in \B(\H)$. Evidently, $\phi$ is a unital $*$-homomorphism. 

To show that $\phi$ is normal, we suppose that $n_\lambda'\in \N'$ is a descending net of positive operators with greatest lower bound $0$, so that $\phi(n_\lambda') \in \B(\H)$ is a descending net of positive operators with some greatest lower bound $x$. Both nets converge in the strong operator topology, so for every vector of the form $\sigma( m)v  \zeta$ we find that $x (\sigma( m) v  \zeta) = \lim_\lambda \phi(n'_\lambda) (\sigma( m) v  \zeta) = \lim_\lambda \sigma( m ) v (n'_\lambda \zeta) = 0$. Since $\H = \overline{\sigma(\M)v\K}$, $x=0$; thus, $\phi$ is normal. Finally, note that $\phi$ is faithful because $v$ is an isometry.

We now have two faithful representations of $\N'$: the canonical representation on $\K$, and the representation $\phi$ on $\H$. By the definition of $\phi$, the Stinespring operator $v$ intertwines these representations in the sense that $\phi(n')(v \zeta) = v n' \zeta$ for all $\zeta \in \K$. Applying Lemma \ref{Completely.C} above, we find a unitary $u\in \B(\ell^2_\kappa \tensor \K, \ell^2_\kappa \tensor \H)$ such that $(\hat e_0^* \tensor 1)u(\hat e_0 \tensor 1) = v$, and $(1 \tensor \phi(n'))u = u (1 \tensor n' )$ for all $n' \in \N'$. We pull the Stinespring representation $\sigma: \M \To \B(\H)$ back along this unitary, i.e., we define $\pi: \M \To \B(\ell^2_\kappa \tensor \K)$ by $\pi(m) = u^*(1 \tensor \sigma(m)) u$. It remains to show first that $\pi(\M) \subsetof \B(\ell^2_\kappa) \stensor \N$, and second that $(\langle e_0|  (\cdot)e_0\rangle\stensor 1) \circ \pi = \psi$.

To prove the first claim, we appeal to the commutation theorem. Fix $m \in \M$, and $n' \in \N'$. By definition of $\phi$, $\phi(n')$ commutes with $\sigma(m)$. Conjugating by $u$, we find that $1 \tensor n'$ commutes with $\pi(m)$. Therefore, $\pi(\M) \subsetof (\CC \stensor \N')' = \B(\ell^2_\kappa) \stensor \N$.

To prove the second claim, we note that since both $u$ and $v$ are isometries, the equation $(\hat e_0^* \tensor 1) u (\hat e_0 \tensor 1) = v$ implies that $(\hat e_0 \hat e_0^* \tensor 1)u (\hat e_0 \tensor 1) = u (\hat e_0 \tensor 1)$. We now calculate:
\begin{align*}
(\langle e_0|  (\cdot)e_0\rangle\stensor 1)(\pi(m))
&=
(\hat e_0^* \tensor 1)u^*(1 \tensor \sigma(m))u(\hat e_0 \tensor 1)
\\ &=
(\hat e_0^* \tensor 1)u^*(  1 \tensor \sigma(m))(\hat e_0 \hat e_0^* \tensor 1 )u(\hat e_0 \tensor 1)
\\ &=
v^* \sigma(m) v
= 
\psi(m)
\end{align*}
\end{proof}

\begin{proof}[Proof of Theorem \ref{Completely.A}]
Let $\psi: \M \To \N$ be a unital normal completely positive map. By Lemma \ref{Completely.C} above, there exist a Hilbert space $\L$, a morphism $\pi: \M \To \B(\L) \stensor \N$, and a state $\mu: \B(\H) \To \CC$, such that $(\mu \stensor 1) \circ \pi = \psi$. We now apply the universal property of $\varepsilon: \M \To \M^{\ast \N} \stensor \N$ to find a morphism $\rho: \M^{\ast \N} \To \B(\L)$ such that $(\rho \stensor 1) \circ \varepsilon = \pi$. The trivial computation $\psi = (\mu \stensor 1) \circ \pi = ((\mu\circ \rho) \stensor 1)\circ \varepsilon$ then shows that $\mu \circ \rho$ does the trick.
\end{proof}

\begin{corollary}
Let $\R$ be a $W^*$-algebra. If $\psi: \M \To \R \stensor \N$ is a unital normal completely positive map, then there exists a unital normal completely positive map $\varphi: \M^{\ast \N} \To \R$ such that $\psi =  (\varphi \stensor 1) \varepsilon$.
\end{corollary}

\begin{proof}
Applying Theorem \ref{Completely.A}, we find a state $\mu: \M^{\R \stensor \N} \To \CC$ such that $\psi = (\mu \stensor 1_{\R \stensor \N}) \circ \varepsilon_{\M, \R \stensor \N}$. By the definition of $\M^{\ast \N}$, there's a unique morphism $\rho: \M^{\ast\N} \To \M^{\ast \R \stensor \N} \stensor \R$ such that $\varepsilon_{\M, \R \stensor \N} =(\rho \stensor 1_\N) \circ \varepsilon_{\M, \N}  $. Therefore, $\psi = [(\mu\stensor 1_\R) \circ \rho) \stensor 1_\N] \circ \varepsilon_{\M, \N}$.
\end{proof}

\section{Measurement}\label{Measurement}

Consider a version of the famous Stern-Gerlach experiment, in which we first prepare an electron in a spin-up state $\mu_\uparrow$, and then measure its spin along the $x$-axis.  We formalize these two steps as quantum operations. The first of these operations is $\o \mu_\uparrow:\{\ast\} \To \o M_2(\CC)$, which prepares the state $\mu_\uparrow$ ``from scratch''. The second operation is $\o\pi: \o M_2(\CC) \To \{-, + \}$, where if $\delta_-$ and $\delta_+$ denote the two nontrivial projections of $\CC^{\{-,+\}}$, then $\pi(\delta_-)$ and $\pi(\delta_+)$ are projections onto the negative spin and positive spin subspaces respectively. This second operation sets a classical two-state system, perhaps a lightbulb, based on the qubit system $\o M_2(\CC)$, the spin of the electron. Composing these two transformations, we obtain a morphism $\o(\mu_\uparrow \circ \pi): \{\ast\} \To \{-, +\}$ in $\o \pWstar$, i.e., a probability measure on the set $\{- ,+\}$. The probability of both outcomes is $\frac 1 2$, in agreement with quantum theory and experiment.

A number of observations are in order. First, each of the three systems discussed above are obtained by specifying a small subalgebra of observables from a much larger algebra of observables of a much larger system, which encompasses the electron, the lightbulb and eventually the experimenter, his lab, and his galaxy. Thus, the trivial quantum system $\o \CC$ is not isolated in some mysterious apparatus in the physicist's lab, but is rather the universal system all of whose states have been identified. Similarly, the spin system $\o M_2(\CC)$ consists of even less than an electron; its non-spin properties are being ignored.

Second, the preparation of the initial state $\mu_\uparrow$ might be further broken down: The electron is emitted by a source in a spin state of maximum entropy, and is then processed via a Stern-Gerlach apparatus, the lower of whose branches ends in an adsorptive stopper. An emitted electron may disappear down the lower branch, never to be heard from again; the apparatus may fail to set the target spin system altogether. The operation induced by the stopper can be represented by the zero $*$-homomorphism $0 \To M_2(\CC)$, whose formal opposite is something like the empty partial function. In general, we define a partial function from $\o \N$ to $\o \M$ to be simply a normal $*$-homomorphism $\M \To \N$.  A probabilistic partial function from $\o \N$ to $\o \M$ is then either a contractive normal completely positive map, or just any normal completely positive map, depending on one's intentions.

Third, according to the language of this paper, the measurement operation $\o\pi: \o M_2(\CC) \To \o \CC^{\{-,+\}}$ is deterministic. The reader may find this conclusion counterintuitive because quantum physics is notoriously nondeterministic, as the Stern-Gerlach experiment demonstrates.  The above description of the Stern-Gerlach experiment blames this circumstance on the preparation of the initial state. Resorting to the vague term `configuration', we might summarize this situation by saying that each configuration of $\o M_2(\CC)$ has a well defined spin in any given direction, but it is impossible to manipulate this system in such a way that we can be certain that it is in a particular configuration. Though this language embraces a kind of realism, it is certainly not the realism of the Einstein, Podolsky and Rosen. The Kochen-Specker Theorem \cite{KochenSpecker67} can easily be applied to show that, in general, it is impossible to assign values to the observables of a quantum system in a way that respects the composition of quantum operations.

\section{Appendix: The Normal Stinespring Theorem}\label{Appendix: The Normal Stinespring Theorem}

For reference, we include a proof of Stinespring's Theorem for normal maps. Stinespring proved the theorem that bears his name in his article \emph{Positive Functions on $C^*$-algebras} \cite{Stinespring55}, which didn't look at $W^*$-algebras, and therefore doesn't show that in this context the Stinespring representation is normal.

\begin{definition}
Let $\M$ and $\N$ be $W^*$-algebras. A $\CC$-linear map $\psi: \M \To \N$
\begin{itemize}
\item is \emph{normal} if it is $w^*$-$w^*$ continuous, and
\item is \emph{completely positive} if for any positive $n \times n$ matrix $[m_{ij}]$ with $m_{ij} \in \M$, the matrix $[\psi(m_{ij})]$ is also positive. 
\end{itemize}
\end{definition}

\begin{theorem}[Stinespring's Theorem for Normal Maps]\label{Normal.B}
Let $\M$ be a $W^*$-algebra, and let $\K$ be a Hilbert space. If $\psi: \M \To \B(\K)$ is a normal completely positive map, then there exist
\begin{enumerate}
\item
a Hilbert space $\H$,
\item
a normal unital $*$-homomorphism $\sigma: \M \To \B(\H)$, and
\item
a bounded operator $v \in \B(\K, \H)$
\end{enumerate}
such that $\psi(m) = v^* \sigma(m) v$ for all $m \in \M$.
\end{theorem}

\begin{proof}
Let $\CC|\M|$  and $\CC |\K|$ denote the complex vector spaces formally spanned the elements of $\M$ and $\K$ respectively, and let $\CC|\M| \tensor \CC |\K|$ denote their algebraic tensor product. The equation $$\langle m_0 \tensor \zeta_0 | m_1 \tensor \zeta_1\rangle_\H = \langle \zeta_0| \psi(m_0^*m_1) \zeta_1\rangle$$ trivially defines a symmetric sesquilinear form on $\CC|\M| \tensor \CC |\K|$.

The complete positivity of $\psi$ then implies that the sesquilinear form $\langle \cdot | \cdot\rangle_\H$ is positive. For any element $\sum_i \alpha_i m_i \tensor \zeta_i$ of the tensor product $\CC|\M| \tensor \CC |\K|$,
\begin{align*}
\left\langle \sum_i \alpha_i m_i \tensor \zeta_i \middle| \sum_i \alpha_i m_i \tensor \zeta_i\right\rangle_\H  = \sum_{ij} \overline \alpha_i \alpha_j \langle \zeta_i | \psi(m_i^*m_j)\zeta_j\rangle \geq 0
\end{align*}
because the matrix $[\psi(m_i^*m_j)]$ is positive. As usual, we complete the quotient of $\CC|\M| \tensor \CC|\K|$ by the subspace of norm-zero vectors, and obtain the Hilbert space $\H$.

An analogous computation shows that for every operator $m \in \M$, and any vector of the form $\sum_i \alpha_i m_i \tensor \zeta_i$ in the tensor product $\CC|\M| \tensor \CC |\K|$, 
$$
\|m\|^2\left\langle \sum_i \alpha_i m_i \tensor \zeta_i \middle| \sum_i \alpha_i m_i \tensor \zeta_i\right\rangle_\H
- \left\langle \sum_i \alpha_i (m m_i) \tensor \zeta_i \middle| \sum_i \alpha_i (mm_i) \tensor \zeta_i\right\rangle_\H \geq 0,
$$
so the equation $\sigma(m)(m_0 \tensor \zeta_0) = (mm_0) \tensor \zeta_0$ defines a bounded operator on $\H$. It is straightforward to show that $\sigma: \M \To \B(\H)$ is a unital $*$-homomorphism.

To show that $\sigma$ is normal, let $m_\lambda \To 0$ be a descending net whose greatest lower bound is $0$. The net $\sigma(m_\lambda) \in \B(\H)$ is also a descending net of positive operators and therefore has a greatest lower bound $x$. In particular, $\sigma(m_\lambda)$ converges to $x$ in the strong operator topology, and $m_\lambda^2$ converges to $0$ in the $w^*$-topology. For every vector of the form $ m \tensor  \zeta \in \H$,
\begin{align*}
\|x ( m \tensor  \zeta)\|^2
& =
\|(\mathop{\mathrm{lim}}^{SOT}_{\lambda \To \infty} \sigma(m_\lambda))( m \tensor  \zeta)\|^2
=
\| \lim_{\lambda \To \infty} (m_\lambda  m) \tensor  \zeta  \|^2
\\ & =
\lim_{\lambda \To \infty} \| (m_\lambda  m ) \tensor  \zeta\|^2
= 
\lim_{\lambda \To \infty} \langle  \zeta | \psi( m m_\lambda^2  m)  \zeta\rangle
=
0
\end{align*}
Since the span of vectors of the form $ m \tensor  \zeta$ is dense in $\H$, we conclude that $x =0$, i.e., that $\sigma$ is normal.

Finally, define $v: \K \To \H$ by $\zeta \mapsto 1 \tensor \zeta$. Straightforward calculation shows that $v$ is a bounded linear operator. It does the trick: For all $m \in \M$ and $\zeta_0, \zeta_1 \in \K$,
$$\langle \zeta_0 | (v^* \sigma(m)v)\zeta_1\rangle = \langle  1 \tensor \zeta_0| \sigma(m)(1 \tensor \zeta_1) \rangle = \langle \zeta_0| \psi(m) \zeta_1\rangle.$$
\end{proof}

\end{document}